\documentclass[11pt,a4paper,reqno,draft]{amsart}

\usepackage{amssymb,amsmath}
\usepackage{amsthm}
\usepackage{latexsym}
\usepackage{graphics}
\usepackage{euscript}
\usepackage[dvips]{graphicx}
\usepackage{mathrsfs}
\usepackage{mathabx}
\usepackage{tikz}
\usepackage{pgf}
\usetikzlibrary{arrows,decorations.pathmorphing,backgrounds,positioning,fit}
\newtheorem{thm}{Theorem}[section]
\newtheorem{prop}[thm]{Proposition}

\newtheorem{lem}[thm]{Lemma}

\newtheorem{defn}[thm]{Definition}

\newtheorem{sexample}[thm]{Example}

\newtheorem{examples}[thm]{Examples}
\newtheorem{squestion}[thm]{Question}

\newtheorem{sremark}[thm]{Remark}
\newtheorem{remarks}[thm]{Remarks}

{\em}

{\em}

\newcounter{substep}
\def\thesubstep{\arabic{substep}}

{\em}

\newcounter{subsubstep}
\def\thesubsubstep{\arabic{subsubstep}}

{\em}

\numberwithin{equation}{section}

\newcommand{\K}{{\mathbb K}} 
\newcommand{\Z}{{\mathbb Z}} \newcommand{\R}{{\mathbb R}}
 \newcommand{\C}{{\mathbb C}}

\newcommand{\sph}{{\mathbb S}} 
 \newcommand{\PP}{{\mathbb P}}

\newcommand{\cl}{\operatorname{Cl}}

\newcommand{\Jac}{\operatorname{Jac}}
\newcommand{\im}{\operatorname{im}}
\newcommand{\Bl}{\operatorname{Bl}}
\newcommand{\hcd}{\operatorname{hcd}}
\newcommand{\supp}{\operatorname{supp}}

\newcommand{\x}{{\tt x}} \newcommand{\y}{{\tt y}}
\newcommand{\z}{{\tt z}} \renewcommand{\t}{{\tt t}}
 \renewcommand{\u}{{\tt u}} 

\newcommand{\ol}{\overline}

\newcommand{\veps}{\varepsilon}

\title[On the set of points at infinity of a polynomial image of $\R^n$]{On the set of points at infinity\\ of a polynomial image of $\R^n$}

\date{}

\author{Jos\'e F. Fernando}
\address{Departamento de \'Algebra, Facultad de Ciencias Matem\'aticas, Universidad Complutense de Madrid, 28040 MADRID (SPAIN)}
\email{josefer@mat.ucm.es}
\thanks{First author supported by Spanish GR MTM2011-22435 and GAAR Grupos UCM 910444.\\
\indent This article is based on a part of the doctoral dissertation of the second author, which has being written under the supervision of the first.}

\author{Carlos Ueno}
\address{Departamento de Matem\'aticas, IES La Vega de San Jos\'e, Paseo de San Jos\'e, s/n, Las Palmas de Gran Canaria, 35015 LAS PALMAS (SPAIN).
}
\email{cuenjac@gmail.com}
\begin{document}

\begin{abstract}
In this work we prove that the \em set of points at infinity $S_\infty:=\cl_{\R\PP^m}(S)\cap\mathsf{H}_\infty$ \em of a semialgebraic set $S\subset\R^m$ that is the image of a polynomial map $f:\R^n\to\R^m$ is connected. This result is no longer true in general if $f$ is a regular map. However, it still works for a large family of regular maps that we call \em quasi-polynomial maps\em.
\end{abstract}
\date{Revised version: 10/06/2014}

\maketitle

\section{Introduction}\label{s1}

A map $f:=(f_1,\ldots,f_m):\R^n\to\R^m$ is a {\em polynomial map} if each component $f_i\in\R[\x]:=\R[\x_1,\ldots,\x_n]$. A subset $S$ of $\R^m$ is a {\em polynomial image} of $\R^n$ if there exists a polynomial map $f:\R^n\to\R^m$ such that $S=f(\R^n)$. More generally, the map $f$ is {\em regular} if each component $f_i$ is a regular function of $\R(\x):=\R(\x_1,\ldots,\x_n)$, that is, $f_i:=\frac{g_i}{h_{i}}$ is a quotient of polynomials such that the zero set of $h_{i}$ is empty. Analogously, a subset $S$ of $\R^m$ is a {\em regular image} of $\R^n$ if it is the image $S=f(\R^n)$ of $\R^n$ given by a regular map $f$. 

The present work continues the general study of polynomial and regular images of Euclidean spaces already began in \cite{fg1,fg2}. A celebrated Theorem of Tarski-Seidenberg \cite[1.4]{bcr} says that the image of any polynomial map (and more generally of a regular map) $f:\R^m\rightarrow\R^n$ is a semialgebraic subset $S$ of $\R^n$, that is, it can be written as a finite boolean combination of polynomial equations and inequalities, which we will call a \em semialgebraic \em description. By \em elimination of quantifiers \em $S$ is semialgebraic if it has a description by a first order formula \em possibly with quantifiers\em. Such a freedom gives easy semialgebraic descriptions for topological operations: interiors, closures, borders of semialgebraic sets are again semialgebraic.

In an \em Oberwolfach \em week \cite{g} Gamboa proposed to characterize the semialgebraic sets of $\R^m$ that are polynomial images of $\R^n$ for some $n\geq1$. The open ones deserve a special attention in connection with the real Jacobian Conjecture \cite{j1,j2,p}. The interest of polynomial (and also regular) images arises because there exist certain problems in Real Algebraic Geometry that can be reduced for such sets to the case $S=\R^n$ (see \cite{fu1,fu2}). Examples of such problems are:
\begin{itemize}
\item[$\bullet$]Optimization of polynomial (and/or regular) functions on $S$, 
\item[$\bullet$]Characterization of the polynomial (or regular functions) that are positive semidefinite on $S$ (Hilbert's 17th problem and Positivstellensatz),
\item[$\bullet$]Computation of trajectories inside $S$ parametrizable by polynomial (or regular) maps.
\end{itemize}

\subsection{Main result}
We denote the projective space of coordinates $(x_0:x_1:\cdots:x_m)$ with $\R\PP^m$. It contains $\R^m$ as the set of points with $x_0=1$. The hyperplane at infinity $\mathsf{H}_\infty$ has equation $x_0=0$. Given a semialgebraic set $S\subset\R^m$, the \em set of points at infinity of $S$ \em is $S_\infty:=\cl_{\R\PP^m}(S)\cap\mathsf{H}_\infty$. Our main result in this work is the following.

\begin{thm}\label{main1}
Let $f:\R^n\to\R^m$ be a non-constant polynomial map and denote $S:=f(\R^m)$. Then $S_\infty$ is non-empty and connected.
\end{thm}

It seems a difficult matter to provide a full geometric characterization of all polynomial and/or regular images $S\subset\R^m$. We only know it for the $1$-dimensional case \cite{fe}. Even though, we have approached the representation as polynomial or regular images of ample families of $n$-dimensional semialgebraic sets whose boundaries are piecewise linear. We have focused on: convex polyhedra, their interiors, their exteriors and the closure of their exteriors \cite{fgu3,fu1,fu2,u2}. The proofs are constructive but the arguments are developed \em ad hoc\em. Two main difficulties arise:
\begin{itemize}
\item To develop a strategy to produce an either polynomial or regular map whose image is the desired semialgebraic set.
\item To prove the surjectivity of the constructed map.
\end{itemize}

In \cite{fg1} appear some straightforward properties that a polynomial (resp. regular) image $S\subset\R^m$ must satisfy: $S$ must be pure dimensional, connected, semialgebraic and its Zariski closure must be irreducible. It follows from \cite[3.1]{fg3} that $S$ must be \em irreducible \em in the sense proposed in \cite{fg3}. All these properties follow readily from the fact \cite[3.6]{fgu2}: 

\vspace{2mm}
\noindent $(*)$ \em Given two points $p,q\in S$, there exists a polynomial (resp. regular) image $L$ of $\R$ \em (also known as \em parametric semiline\em) \em contained in $S$ and passing through $p,q$\em. 

\vspace{2mm}
There are many examples of semialgebraic sets with property $(*)$ that are polynomial images of no $\R^n$. Take $S:=\{0\leq x\leq 1, 0\leq y\}\cup\{0\leq y\leq x\}\subset\R^2$, which satisfies $(*)$. By Theorem \ref{main1} $S$ is a polynomial image of no $\R^n$ because its set of points at infinity is disconnected. Consequently Theorem \ref{main1} provides a new obstruction to be a polynomial image of $\R^n$.

We wondered in \cite[7.3]{fg2} about the number of connected components of the exterior of a polynomial image of dimension $\geq2$. The first author was convinced that the answer was one, but the second author showed in \cite{u1} that this number can be arbitrarily large. Nevertheless, Theorem \ref{main1} is in the vein of our wrong initial position. 

\subsection{Strategy of the proof and structure of the article}
The proof of Theorem \ref{main1} involves techniques inspired by those employed by Jelonek in his works \cite{j1,j2} where he studies the geometry of the set of points ${\mathcal S}_f$ at which an either complex or real polynomial map $f:\K^n\to\K^m$ is not proper ($\K$ denotes either $\R$ or $\C$). We highlight the following: 
\begin{itemize}
\item[$\bullet$]\em Resolution of indeterminacy of rational maps \em defined on projective surfaces.
\item[$\bullet$]Sufficient conditions to guarantee that the intersection of two connected complex projective curves of a complex projective surface is either empty or a singleton.
\item[$\bullet$]A \em `rational' curve selection lemma\em. 
\end{itemize}
For the sake of the reader we include a careful exposition of these techniques in Section \ref{s2}. The reader can proceed directly to Section \ref{s3} and refer to the Preliminaries only when needed. In Section \ref{s3} we prove Theorem \ref{main1} in the more general setting of \em quasi-polynomial maps\em. In Section \ref{s4} we show that the set of points at infinity of the image of a general regular map does not need to be connected and we provide some enlightening examples. 

\subsection*{Acknowledgements}
The authors thank Prof. Jes\'us M. Ruiz for many helpful discussions and suggestions during the preparation of this article. The authors are also very grateful to S. Schramm for a careful reading of the final version and for the suggestions to refine its redaction.

\section{Preliminaries}\label{s2}

We write $\K$ to refer indistinctly to $\R$ or $\C$. We denote the hyperplane at infinity of $\K\PP^m$ with $\mathsf{H}_\infty(\K):=\{x_0=0\}$. Clearly, $\K\PP^m$ contains $\K^m$ as the set $\K\PP^m\setminus\mathsf{H}_\infty(\K)=\{x_0=1\}$. If $m=1$, we denote the point at infinity $\K\PP^1$ with $\{p_\infty\}:=\{x_0=0\}$ and if $m=2$, we write $\ell_\infty(\K):=\{x_0=0\}$ for the line at infinity of $\K\PP^2$. We use freely that the real projective space $\R\PP^m$ can be immersed in $\R^k$ for $k$ large enough as an affine non-singular real algebraic variety \cite[3.4.4]{bcr}. Thus, the closure in $\R\PP^m$ of a semialgebraic subset of $\R^m$ is again a semialgebraic set. It will be useful to understand real algebraic objects as fixed parts under conjugation of complex algebraic objects that are invariant under conjugation.

\subsection{Invariant projective objects.}\label{inv}
For each $n\geq1$ denote the complex conjugation with
$$
\sigma:=\sigma_n:\C\PP^n\to\C\PP^n,\ z=(z_0:z_1:\cdots:z_n)\mapsto\ol{z}=(\ol{z_0}:\ol{z_1}:\cdots:\ol{z_n}).
$$
Clearly, $\R\PP^n$ is the set of fixed points of $\sigma$. A set $A\subset\C\PP^n$ is called \em invariant \em if $\sigma(A)=A$. If $Z\subset\C\PP^n$ is an invariant non-singular (complex) projective variety, then $Z\cap\R\PP^n$ is a non-singular (real) projective variety. We say that a rational map $h:\C\PP^n\dashrightarrow\C\PP^m$ is \em invariant \em if $h\circ\sigma_n=\sigma_m\circ h$. Of course, $h$ is invariant if its components are homogeneous polynomials with real coefficients, so it provides by restriction a real rational map $h|_{\R\PP^n}:\R\PP^n\dashrightarrow\R\PP^m$. We use freely usual concepts of Algebraic Geometry such as: rational map, regular map, divisor, blow-up, etc. and refer the reader to \cite{ha,m,sh1,sh2} for further details. For the sake of the reader we denote complex dimension with $\dim_\C(\cdot)$ and real dimension with $\dim_\R(\cdot)$. Recall the following fact concerning the regularity of rational maps defined on a non-singular projective curve \cite[7.1]{m}.

\begin{lem}\label{curve}
Let $Z\subset\C\PP^n$ be a non-singular projective curve and $F:Z\dashrightarrow\C\PP^m$ a rational map. Then $F$ extends to a regular map $F':Z\to\C\PP^m$. In addition, if $Z,F$ are invariant, so is $F'$.
\end{lem}

One of the main tools is the resolution of indeterminacy of an invariant rational map. We provide a careful presentation of this well-known tool taking care of invariance.

\subsection{Resolution of indeterminacy of an invariant rational map.}\label{res}
Let $Z_0\subset\C\PP^n$ be an invariant non-singular projective variety of dimension $d$ and
$$
F_\C:=(F_1:\cdots:F_m):Z_0\dashrightarrow\C\PP^m
$$ 
an invariant rational map. To compute the \em set of indeterminacy \em of $F_\C$ one proceeds as follows \cite[III.1.4]{sh1}. Consider for each $i=0,\ldots,m$ the divisor $D_i$ in $Z_0$ defined by $F_i$ and let $|D|:=\hcd\{D_0,D_1,\ldots,D_m\}$ be the highest common divisor of the divisors $D_0,D_1,\ldots,D_m$. The divisors $D_i':=D_i-|D|$ are relatively prime. By \cite[III.1.4.Thm.2]{sh1} the map $F_\C$ fails to be regular exactly at the points of the invariant set $Y_\C:=\bigcap_{i=0}^m\supp(D_i')$, which has dimension $\leq d-2$. As $F_\C$ is invariant, it can be restricted to a real rational map $F_\R:Z_0\cap\R\PP^n\dashrightarrow\R\PP^m$ whose set of indeterminacy is $Y_\R:=Y_\C\cap\R\PP^2$. 

We assume that $Z_0$ has dimension $2$. As it is well-known, $F_\C:Z_0\dashrightarrow\C\PP^m$ admits an invariant resolution. Namely,

\paragraph{}\label{clue} There exist:
\em\begin{itemize}
\item[(i)] An invariant non-singular projective surface $Z_1\subset\C\PP^k$ for some $k\geq2$. 
\item[(ii)] An invariant (composition of a) sequence of blow-ups $\pi_\C:Z_1\to Z_0\subset\C\PP^n$ such that $\pi_\C|_{Z_1\setminus\pi_\C^{-1}(Y_\C)}:Z_1\setminus\pi_\C^{-1}(Y_\C)\to Z_0\setminus Y_\C$ is a biregular isomorphism and 
$$
Y_\C=\{p\in Z_0:\ \#\pi_\C^{-1}(p)>1\}.
$$
\item[(iii)] An invariant regular map $\widehat{F}_\C:Z_1\to\C\PP^m$ such that 
$$
\widehat{F}_\C|_{Z_1\setminus\pi_\C^{-1}(Y_\C)}=F_\C\circ\pi_\C|_{Z_1\setminus\pi_\C^{-1}(Y_\C)}.
$$
\end{itemize}

In addition, for each $y\in Y_\C$ the irreducible components of $\pi_\C^{-1}(y)$ are non-singular projective curves $K_{i,y}$ that are biregularly equivalent to $\C\PP^1$ (via regular maps $\Phi_{i,y}:\C\PP^1\to K_{i,y}$ that are invariant for invariant $K_{i,y}$) and satisfy:
\begin{itemize}
\item[(iv)] If $y\in Y_\C\setminus Y_\R$, then $\sigma(K_{i,y})=K_{i,\sigma(y)}$ and $K_{i,y}\cap\R\PP^k=\varnothing$,
\item[(v)] If $y\in Y_\R$, then either $K_{i,y}\cap\R\PP^k=\varnothing$ and there exists $j\neq i$ such that $\sigma(K_{i,y})=K_{j,y}$ or $\sigma(K_{i,y})=K_{i,y}$ and $C_{i,y}=K_{i,y}\cap\R\PP^k$ is a non-singular projective curve biregularly equivalent to $\R\PP^1$ (via $\phi_{i,y}:=\Phi_{i,y}|_{\R\PP^1}:\R\PP^1\to C_{i,y}$).
\end{itemize} 
\em A triple $(Z_1,\pi_\C,\widehat{F}_\C)$ satisfying the previous properties is an \em invariant resolution for $F_\C$\em.

Let us recall some terminology and results concerning blow-ups of non-singular projective varieties at non-singular centers, from which \ref{clue} follows readily.

\subsection{Blow-up with a non-singular variety as center}\label{blu}\setcounter{paragraph}{0} 

Let $Z_0\subset\C\PP^n$ be a non-singular irreducible projective variety and $Y\subset Z_0$ a non-singular subvariety. Let $H_1,\ldots,H_m$ be a system of homogeneous polynomials of the same degree that generates an ideal $I$ whose saturation
$$
\ol{I}:=\{H\in\C[\z]:=\C[\z_0,\z_1,\ldots,\z_n]:\ (\z)^kH\subset I\quad\text{for some $k\geq0$}\}
$$
equals the ideal ${\mathcal J}(Y)$ of (homogeneous) polynomials of $\C[\z]$ vanishing identically on $Y$.

\paragraph{}\label{propblu} The \em blow-up $\Bl_{Y}(Z_0)$ of $Z_0$ with center $Y$ \em is the closure in $Z_0\times\C\PP^{m-1}$ of the set
$$
\{(z;(H_1(z):\cdots:H_m(z)))\in (Z_0\setminus Y)\times\C\PP^{m-1}\}
$$
together with the projection $\pi:\Bl_{Y}(Z_0)\subset Z_0\times\C\PP^{m-1}\to Z_0,\ (z;y)\mapsto z$ (see \cite[7.18]{ha} and \cite[VI.2.2]{sh2}). Recall the following facts: 
\begin{itemize}
\item[$\bullet$]$\Bl_{Y}(Z_0)$ is a non-singular irreducible projective variety of the same dimension as $Z_0$, independent of the choices made in the process.
\item[$\bullet$]$\pi^{-1}(Y)$ is a non-singular hypersurface of $\Bl_{Y}(Z_0)$.
\item[$\bullet$]$\pi|_{\Bl_{Y}(Z_0)\setminus\pi^{-1}(Y)}:\Bl_{Y}(Z_0)\setminus\pi^{-1}(Y)\to Z_0\setminus Y$ is a biregular isomorphism.
\item[$\bullet$]$Y$ admits a finite cover by affine open subsets $\{U_\alpha\}_\alpha$ satisfying $\pi^{-1}(U_\alpha)\cong U_\alpha\times\C\PP^{r-1}$ where $r:=\dim_\C(Z_0)-\dim_\C(Y)$. In particular, the fiber of each $y\in Y$ is a projective space $\C\PP^{r-1}$.
\item[$\bullet$]If $Y_1,\ldots,Y_r$ are the irreducible components of $Y$, then they are pairwise disjoint and non-singular and it holds
$$
\Bl_Y(Z_0)\cong\Bl_{Y_r}(\cdots\Bl_{Y_2}(\Bl_{Y_1}(Z_0))\cdots).
$$
\end{itemize}
If $Z_0,Y$ are invariant, $\Bl_{Y}(Z_0)$ can be assumed invariant, too, by choosing an ideal $I$ with saturation ${\mathcal J}(Y)$ and whose generators are invariant (given any family of generators, consider the real and the imaginary parts of all of them). If we consider the immersion of $\Bl_{Y}(Z_0)$ in some $\C\PP^N$ using Segre's map, also the regular map $\pi:\Bl_{Y}(Z_0)\to Z_0$ is invariant. 

\paragraph{}\label{curve2} Assume that $Z_0$ is an invariant non-singular projective surface and $Y$ is a finite invariant subset. Consider the invariant blow-up $(\Bl_{Y}(Z_0),\pi)$ of $Z_0$ with center $Y$:

\indent (i) For each $y\in Y$ the fiber $\pi^{-1}(y)$ is a $\C\PP^1$. If $y\in Y\cap\R\PP^d$, there exists an invariant biregular equivalence between $\pi^{-1}(y)$ and $\C\PP^1$. If $y\in Y\setminus\R\PP^d$, then $\pi^{-1}(y)\cap\R\PP^N=\varnothing$.

\indent (ii) If $C\subset Z_0$ is a non-singular curve not contained in the center $Y$, its strict transform 
$$
\widetilde{C}:=\cl_{\Bl_{Y}(Z_0)}(\pi^{-1}(C\setminus Y))\subset\pi^{-1}(C)
$$ 
is as well a non-singular curve. In addition, if there exists an invariant biregular equivalence $\Phi:\C\PP^1\to C$, the strict transform $\widetilde{C}\subset\pi^{-1}(C)$ of $C$ under $\pi$ is invariant and there exists an invariant biregular equivalence $\Psi:\C\PP^1\to\widetilde{C}$.

\begin{proof}[Sketch of proof of statement \em \ref{clue}]
To solve the indeterminacy of the rational map $F_\C:Z_0\dashrightarrow\C\PP^m$, one blows the set of points of indeterminacy $Y_\C$ of $F_\C$ up and considers the composition $G$ of $F_\C$ with the previous sequence of blowing-ups next \cite[IV.3.3.Thm.3]{sh1}. If $G$ is regular, the process is concluded. Otherwise one applies the previous procedure to $G$. In finitely many steps one achieves a regular map and the process finishes. By \ref{blu} we may assume that each rational map involved in the process is invariant, so in the last step of the process we obtain: 
\begin{itemize}
\item An invariant non-singular projective surface $Z_1\subset\C\PP^k$ for some $k\geq2$, 
\item An invariant sequence of blow-ups $\pi_\C:Z_1\to Z_0\subset\C\PP^n$ and
\item An invariant regular map $\widehat{F}_\C:Z_1\to\C\PP^m$ such that 
$$
\widehat{F}_\C|_{Z_1\setminus\pi_\C^{-1}(Y_\C)}=F_\C\circ\pi_\C|_{Z_1\setminus\pi_\C^{-1}(Y_\C)}.
$$
\end{itemize}
The triple $(Z_1,\pi_\C,\widehat{F}_\C)$ satisfies conditions (i) to (iii). The fiber of each point of $Y_\C$ under $\pi$ is a complex projective curve by \ref{blu} whose irreducible components are non-singular rational curves while the fiber of each point of $Z_0\setminus Y_\C$ under $\pi$ is a singleton. Thus, $Y_\C$ is the fundamental set of $\pi_\C$, so (ii) holds. Assertions (iv) and (v) are straightforward consequences of \ref{blu}.
\end{proof}

\subsection{Projective curves intersecting each other in a singleton.}\setcounter{paragraph}{0}
It will be useful to know sufficient conditions that guarantee that the intersection of two connected complex projective curves of a complex projective surface is either empty or a singleton. The proof of the following result is deeply inspired by the proofs of \cite[4.6]{j1} and \cite[3.1]{j2}.

\begin{lem}\label{crack}
Let $X$ be a complex projective surface. Assume that: 
\begin{itemize}
\item[$\bullet$]$U\subset X$ is a connected orientable manifold such that 
$$
H_1(U;\Z)=H_2(U;\Z)=0,
$$
\item[$\bullet$]$U$ is dense in $X$ and the complement $A:=X\setminus U$ is a complex projective curve. 
\end{itemize}
Let $C_1,C_2\subset A$ be two connected, complex projective curves without common irreducible components. Then the intersection $C_1\cap C_2$ is either the empty set or a singleton.
\end{lem}
\begin{proof}
The proof is conducted in several steps:

\paragraph{} We prove first: $H^{1}(A;\Z)=0$.

Assume that $X$ is a compact polyhedron of dimension $4$ and $A$ a closed subpolyhedron of $X$. As $U$ is an orientable real manifold of dimension $4$, by Lefschetz duality \cite[6.1.11 \& 6.2.19]{sp} we have $H^{i}(X,A;\Z)\cong H_{4-i}(U;\Z)$ for $i=0,\ldots,4$. By the long exact sequence of cohomology \cite[5.4.13]{sp}
$$
H^1(X,A;\Z)\to H^1(X;\Z)\to H^{1}(A;\Z)\to H^2(X,A;\Z)\cong H_2(U;\Z)=0,
$$
so $H^1(X;\Z)\to H^{1}(A;\Z)$ is an epimorphism. As $U$ is a connected open dense subset of $X$, there exists by \cite[4.7]{j1} an epimorphism $0=H_1(U;\Z)\to H_1(X,\Z)$, so $H_1(X,\Z)=0$. By the universal-coefficient theorem for cohomology \cite[5.5.3]{sp} 
$$
H^1(X,\Z)\cong{\rm Hom}(H_1(X,\Z),\Z)\oplus{\rm Ext}(\Z,\Z),
$$
so by \cite[5.5.1]{sp} $H^{1}(X;\Z)=0$. Consequently, $H^{1}(A;\Z)=0$.

\paragraph{}\label{retract}\em Let $C\subset A$ be a projective algebraic curve. Then $H^1(C;\Z)=0$\em. In particular, it holds $H^1(C_1\cup C_2;\Z)=0$.

Let $C'$ be the union of the irreducible components of $A$ not contained in $C$. Clearly, ${\mathcal F}:=C\cap C'$ is a finite set. As $C$ and $C'$ are analytic sets, they are locally contractible, so for each $x\in{\mathcal F}$ there exists a neighborhood $V^x$ in $A$ such that: 
\begin{itemize}
\item[$\bullet$]$V^x\cap C$ and $V^x\cap C'$ have the singleton $\{x\}$ as a deformation retract, 
\item[$\bullet$]$V^x\cap C\cap C'=\{x\}$ and 
\item[$\bullet$]$V^{x_1}\cap V^{x_2}=\varnothing$ if $x_1,x_2\in{\mathcal F}$ and $x_1\neq x_2$. 
\end{itemize}
It holds that $V:=C\cup\bigcup_{x\in{\mathcal F}}(V^x\cap C')$ and $W:=C'\cup\bigcup_{x\in{\mathcal F}}(V^x \cap C)$ are open subsets of $A$ such that $V\cup W=A$, $V\cap W=\bigcup_{x\in{\mathcal F}}(V^x \cap (C\cup C'))$ and $C,C'$ are respective deformation retracts of $V$ and $W$. In addition, ${\mathcal F}$ is a deformation retract of $V\cap W$. By Mayer-Vietoris' exact sequence for cohomology \cite[5.4.9]{sp}
\begin{multline*}
0=H^1(A;\Z)=H^1(V\cup W;\Z)\to H^1(V;\Z)\oplus H^1(W;\Z)\\
\to H^1(V\cap W;\Z)\cong H^1({\mathcal F};\Z).
\end{multline*}
As ${\mathcal F}$ is a finite set, $H^1({\mathcal F};\Z)=0$, so $H^1(C;\Z)\cong H^1(V;\Z)=0$. 

\paragraph{}\em$C_1\cap C_2$ is either empty or a singleton\em.

Assume $C_1\cap C_2\neq\varnothing$. Let $V_1,V_2$ be two open subsets of $C_1\cup C_2$ such that 
\begin{itemize}
\item[$\bullet$]$C_i\subset V_i$ is a deformation retract of $V_i$ for $i=1,2$, 
\item[$\bullet$]$V_1\cup V_2=C_1\cup C_2$ and $C_1\cap C_2$ is a deformation retract of $V_1\cap V_2$.
\end{itemize}
(for the construction of $V_1,V_2$ proceed similarly to \ref{retract}). By Mayer-Vietoris' exact sequence for reduced cohomology \cite[5.4.8 \& p.240]{sp} applied to the open subsets $V_1$ and $V_2$ of $C_1\cup C_2$ (whose intersection $V_1\cap V_2\supset C_1\cap C_2\neq\varnothing$), we deduce
\begin{multline*}
\tilde{H}^{0}(C_1;\Z)\oplus\tilde{H}^{0}(C_2;\Z)\cong\tilde{H}^{0}(V_1;\Z)\oplus\tilde{H}^{0}(V_2;\Z)\\
\to\tilde{H}^{0}(V_1\cap V_2;\Z)\cong\tilde{H}^{0}(C_1\cap C_2;\Z)\to\tilde{H}^1(V_1\cup V_2;\Z)\\
\cong\tilde{H}^1(C_1\cup C_2;\Z)\cong H^1(C_1\cup C_2;\Z)=0. 
\end{multline*}
As $C_1,C_2$ are connected, $\tilde{H}^{0}(C_i;\Z)=0$, so $\tilde{H}^{0}(C_1\cap C_2;\Z)=0$. Thus, the finite set $C_1\cap C_2$ is connected, so it is a singleton.
\end{proof}
\begin{sexample}\label{bouquet}\em
Let ${\mathcal F}\subset\C^2$ be a finite set and $U:=\C^2\setminus{\mathcal F}$ its complement. Then $H_1(U,\Z)=H_2(U,\Z)=0$. 

By Hurewicz's theorem $H_1(U,\Z)$ is the abelianization of $\pi_1(U)=0$, so $H_1(U,\Z)=0$. We identify $\C^2\equiv\R^4$. To compute $H_2(U,\Z)$, we may assume ${\mathcal F}:=\{p_1,\ldots,p_r\}$ where $p_k:=(2k-1,0,0,0)$. Notice that $D_r:=\bigcup_{i=1}^r\sph^3_{p_i}$ where $\sph^3_{p_i}:=\{x\in\R^4:\ \|x-p_i\|=1\}$ is a deformation retract of $U\equiv\R^4\setminus{\mathcal F}$. Observe 
$$
\sph^3_{p_i}\cap\sph^3_{p_j}=\begin{cases}
\{p_{ij}:=\frac{p_i+p_j}{2}\}&\text{if $|i-j|=1$},\\
\varnothing&\text{if $|i-j|>1$}.
\end{cases}
$$
Denote $D_k:=\bigcup_{i=1}^k\sph^3_{p_i}$ and observe $H_2(D_1;\Z)=H_2(\sph^3_{p_1};\Z)=0$. By induction hypothesis, we assume $H_2(D_{r-1};\Z)=0$. By Mayer-Vietoris' exact sequence for homology \cite[\S4.6]{sp}
\begin{multline*}
0=H_2(D_{r-1},\Z)\oplus H_2(\sph^3_{p_r};\Z)\to H_2(D_r;\Z)=H_2(U;\Z)\\
\to H_1(D_{r-1}\cap\sph^3_{p_r},\Z)=H_1(\{p_{r-1,r}\};\Z)=0,
\end{multline*}
so $H_2(U;\Z)=0$, as required.\qed
\end{sexample}

\begin{sremark}\label{crackc}\em
Lemma \ref{crack} applies if $U$ is homeomorphic to the complement in $\C^2$ of a finite subset. 
\end{sremark}

\subsection{Rational curve selection lemma.}\setcounter{paragraph}{0}
To finish this section we present the following variation of the curve selection lemma adapted to the situations we will approach later. We refer the reader to \cite[4.7]{jk} for a result of similar nature.

\begin{lem}\label{polcurve}
Let $f:\R^n\to\R^m$ be a regular map and $S:=f(\R^n)$. Let $S'\subset S$ be a semialgebraic dense subset of $S$ and $p\in\cl_{\R\PP^m}(S)\setminus S$. Then there exist (after reordering the variables of $\R^n$) a rational path $\alpha:=(\pm\t^{k_1},\t^{k_2}{\tt p}_2,\ldots,\t^{k_n}{\tt p}_n)\in\R(\t)^n$ where
\begin{itemize}
\item[$\bullet$]$k_i\in\Z$, $k_1=\min\{k_1,\ldots,k_n\}<0$, 
\item[$\bullet$]${\tt p}_i\in\R[\t]$ and ${\tt p}_i(0)\neq 0$ for $i=2,\ldots,n$ 
\end{itemize} 
and an integer $r\geq1$ such that for each $\beta\in(\t)^r\R[\t]^n$
\begin{itemize} 
\item[(i)] $p=\lim_{t\to0^+}(f\circ(\alpha+\beta))(t)$ and
\item[(ii)] $(f\circ(\alpha+\beta))(t)\subset S'$ for $t>0$ small enough.
\end{itemize}
\end{lem}

Before proving the previous result, we need a technical lemma.
\begin{lem}\label{order}
Let $F\in\R[\x]$ be a polynomial that is not identically zero and let $g\in\R((\t))^n$. Then for each $s\geq1$ there exists $r\geq1$ such that if $h\in(\t)^r\R[[\t]]^n$, we have $F(g)-F(g+h)\in(\t)^s\R[[\t]]$.
\end{lem}
\begin{proof}
Write $g:=\frac{g'}{\t^k}$ where $k\geq 0$ and $g'\in\R[[\t]]^n$. Let $\z$ and $\y:=(\y_1,\ldots,\y_n)$ be variables. Write $F(\x+\z\y)=F(\x)+\z H(\x,\y,\z)$ where $H\in\R[\x,\y,\z]$ is a polynomial of degree $d$. Let $r:=s+kd$ and observe that if $h\in(\t)^r\R[[\t]]^n$, we may write $h:=\t^rh'$ where $h'\in\R[[\t]]$ and
$$
F(g+h)-F(g)=\t^rH\Big(\frac{g'}{\t^k},h',\t^r\Big).
$$
Observe that the order of the series $F(g+h)-F(g)$ is $\geq r-kd=s$, as required.
\end{proof}

\begin{proof}[Proof of Lemma \em\ref{polcurve}]
The proof is conducted in several steps:

\paragraph{} We may assume: \em $S'$ is open in $S$\em.

As $S'$ is dense in $S$ and $S$ is pure dimensional because it is the image of $\R^n$ under a regular map, it holds that $\cl_{\R^{m}}(S\setminus S')$ has dimension $\leq\dim_\R(S)-1$. Thus, $S'':=S\setminus\cl_{\R^m}(S\setminus S')\subset S'$ is dense and open in $S$. Changing $S'$ with $S''$, we may assume that $S'$ is open in $S$.

\paragraph{} \em There exists a Nash path $\lambda:(-1,1)\to\R\PP^n$ such that $(f\circ\lambda)(0,1)\subset S'$, $\lim_{t\to0}\lambda(t)\!=q\in\R\PP^n\setminus\R^n$ and $\lim_{t\to0}(f\circ\lambda)(t)=p$\em.

As $p\in\cl_{\R\PP^m}(S')\setminus S'$, there exists by the Nash curve selection lemma \cite[8.1.13]{bcr} a Nash path $\gamma:(-1,1)\to\R\PP^m$ such that $\gamma((0,1))\subset S'$ and $\gamma(0)=p$. Let $\{x_k\}_k\subset\R^n$ and $\{t_k\}_{k\geq1}\subset(0,1)$ be sequences such that $\lim_{k\to\infty}t_k=0$ and $f(x_k)=\gamma(t_k)$ for all $k\geq1$. We may assume that $\{x_k\}_{k\geq1}$ converges to $q\in\R\PP^n$. As $S=f(\R^n)$ and $p=\gamma(0)\in\cl_{\R\PP^m}(S)\setminus S$, we have $q\in\R\PP^n\setminus\R^n=\mathsf{H}_\infty(\R)$. 

As $\dim_\R(f^{-1}(\gamma((0,1))))\geq\dim_\R(\gamma((0,1)))=1$ and 
$$
q\in\cl_{\R\PP^n}(f^{-1}(\gamma((0,1))))\setminus(f^{-1}(\gamma((0,1)))),
$$ 
there exists a Nash path $\lambda:(-1,1)\to\R\PP^n$ such that $\lambda((0,1))\subset f^{-1}(\gamma((0,1)))$ and $\lambda(0)=q$. 

\paragraph{} \em Construction of the integer $k_1$\em. After reparametrizing $\gamma$, we may assume $f\circ\lambda=\gamma$, so there exist 
$$
\lambda_0,\lambda_1,\ldots,\lambda_n,\gamma_0,\gamma_1,\ldots,\gamma_m\in\R[[\t]]
$$ 
such that 
$$
f\circ\lambda=f\Big(\frac{\lambda_1}{\lambda_0},\ldots,\frac{\lambda_n}{\lambda_0}\Big)=\Big(\frac{\gamma_1}{\gamma_0},\ldots,\frac{\gamma_m}{\gamma_0}\Big)=\gamma,
$$
$\lim_{t\to0^+}(1:\lambda(t))=q$ and $\lim_{t\to0^+}(1:\gamma(t))=p$. As $q\in\mathsf{H}_\infty(\R)$, we may assume after reordering the variables $x_1,\ldots,x_n$ that the order 
$$
k_1:=\omega\Big(\frac{\lambda_1}{\lambda_0}\Big)=\min\Big\{\omega\Big(\frac{\lambda_i}{\lambda_0}\Big):\ i=1,\ldots,n\Big\}<0. 
$$
After a change of the type $\t\mapsto\t u(\t)$ where $u\in\R[[\t]]$ is a unit, we assume $\frac{\lambda_1}{\lambda_0}=\pm\t^{k_1}$ with $k_1<0$. Write 
$$
\frac{\lambda_i}{\lambda_0}=t^{k_i'}\rho_i\quad\text{where }\begin{cases}
\text{$k_i'\in\Z$, $\rho_i\in\R[[\t]]$ and $\rho_i(0)\neq 0$}&\text{if $\frac{\lambda_i}{\lambda_0}\neq0$,}\\
\text{$k_i'=0$ and $\rho_i=0$}&\text{if $\frac{\lambda_i}{\lambda_0}=0$.}
\end{cases}
$$ 

\paragraph{}\em Construction of the integer $r$\em. Write $f:=(\frac{f_1}{f_0},\ldots,\frac{f_m}{f_0})$ where $f_i\in\R[\x]$ and $f_0$ does not vanish on $\R^n$. By Lemma \ref{order} there exists $r_0\geq 1$ such that if $\mu:=(\mu_1,\ldots,\mu_n)\in(\t)^{r_0}\R[[\t]]^n$, then
\begin{multline*}
\lim_{t\to0^+}(1:\lambda(t)+\mu(t))=q\\
\text{and}\quad\lim_{t\to0^+}\Big(1:\frac{f_1}{f_0}(\lambda(t)+\mu(t)):\cdots:\frac{f_m}{f_0}(\lambda(t)+\mu(t))\Big)=p.
\end{multline*}
Since $S'$ is open in $S$, also $f^{-1}(S')$ is open in $\R^n$. As 
$$
\lambda((0,1))\subset f^{-1}(\gamma((0,1)))\subset f^{-1}(S'),
$$ 
there exist finitely many polynomials $g_1,\ldots,g_q\in\R[\x]$ such that 
$$
\lambda((0,\veps_1))\subset\{g_1>0,\ldots,g_q>0\}\subset f^{-1}(S')
$$
for some $\veps_1>0$ small enough. Thus, $g_i\circ\lambda=a_it^{\ell_i}+\cdots$ where $a_i>0$ and $\ell_i\in\Z$. By Lemma \ref{order} there exists $r\geq\max\{r_0,k_2',\ldots,k_n'\}+1$ such that if $\eta\in(\t)^r\R[[\t]]$, then 
$$
(g_i\circ(\lambda+\eta))-(g_i\circ\lambda)\in(\t)^s\R[[\t]]
$$ 
where $s:=\max\{0,\ell_1,\ldots,\ell_q\}+1$. Consequently, if $\eta\in(\t)^r\R[[\t]]$, each series $g_i\circ(\lambda+\eta)>0$ for $t>0$ small enough, so $(\lambda+\eta)(t)\in f^{-1}(S)$ for $t>0$ small enough.

\paragraph{}\em Construction of the rational path $\alpha$\em. Choose $\mu:=(0,\mu_2,\ldots,\mu_n)\in(\t)^r\R[[\t]]^n$ such that $\t^{-k_i'}\,(\frac{\lambda_i}{\lambda_0}+\mu_i)=\rho_i+\t^{-k_i'}\mu_i\in\R[\t]\setminus\{0\}$. Write $\t^{k_i'}(\rho_i+\t^{-k_i'}\mu_i)=\t^{k_i}{\tt p}_i$ where $k_i\in\Z$, ${\tt p}_i\in\R[\t]$ and ${\tt p}_i(0)\neq0$. The rational path $\alpha:=(\pm\t^{k_1},\t^{k_2}{\tt p}_2,\ldots,\t^{k_n}{\tt p}_n)\in\R(\t)^n$ and the integer $r$ satisfy the conditions in the statement.
\end{proof}

\section{Connectedness of the set of points at infinity of a polynomial image}\label{s3}

\subsection{Proof of Theorem \ref{main1}}\label{s3a}\setcounter{paragraph}{0} 

The purpose of this section is to prove Theorem \ref{main1}. We approach this result in the more general framework of \em quasi-polynomial maps\em. Given a regular map 
$$
f:=\Big(\frac{f_1}{f_0},\cdots,\frac{f_m}{f_0}\Big):\R^n\to\R^m
$$ 
where each $f_i\in\R[\x]$, consider the invariant rational map 
$$
F_\C:\C\PP^n\dashrightarrow\C\PP^m,\ x:=(x_0:x_1:\cdots:x_n)\mapsto(F_0(x):F_1(x):\cdots:F_m(x))
$$
where $F_i(\x_0:\x_1:\cdots:\x_n):=\x_0^df_i(\frac{\x_1}{\x_0},\ldots,\frac{\x_n}{\x_0})$ and $d:=\max_{i=0,\ldots,m}\{\deg(f_i)\}$. Let $Y_\C$ be the set of indeterminacy of $F_\C$ and write $F_0=\x_0^eF_0'$ where $e\geq0$ and $F_0'\in\R[\x_0,\x_1,\ldots,\x_n]$ is a homogeneous polynomial that is not divisible by $\x_0$. Observe that $F_\C$ can be restricted to a rational map $F_\R:\R\PP^n\dashrightarrow\R\PP^m$ whose set of indeterminacy is $Y_\R:=Y_\C\cap\R\PP^m$. 

\begin{defn}\label{qp}\em
We say that $f$ is a \em quasi-polynomial map \em if $F_\C(\mathsf{H}^n_\infty(\C))\subset\mathsf{H}_\infty^m(\C)$ and no real indeterminacy point of $F_\C$ belongs to $\{F_0'=0\}$, that is, $Y_\C\cap\R\PP^n\cap\{F_0'=0\}=\varnothing$.
\end{defn}
\begin{remarks}\em
(i) The condition $F_\C(\mathsf{H}^n_\infty(\C))\subset\mathsf{H}^m_\infty(\C)$ is equivalent to 
$$
e:=\max\{\deg(f_1),\ldots,\deg(f_m)\}-\deg(f_0)>0.
$$

(ii) If the polynomials $f_0,f_1,\ldots,f_m$ are relatively prime, then
$$
Y_\C\cap\{F_0'=0\}=\{F_0'=0,F_1=0,\ldots,F_m=0\}.
$$

(ii) Quasi-polynomial maps include polynomial maps, but also more general regular maps. If $e>0$ and $\{F_0'=0\}\cap\R\PP^n=\varnothing$, then $f$ is a quasi polynomial map. This occurs for instance if $e>0$ and $f_0:=b_0^2+(b_1^2\x_1^{2k}+\cdots+b_n^2\x_n^{2k})^\ell$ where $k,\ell\geq0$ and $b_i\in\R\setminus\{0\}$.

(iv) If $n=1$, the condition $e:=\max\{\deg(f_1),\ldots,\deg(f_m)\}-\deg(f_0)>0$ characterizes quasi-polynomial maps.
\end{remarks}

\begin{thm}\label{property}
Let $S\subset\R^m$ be a semialgebraic set that is the image of a quasi-polynomial map $f:\R^n\to\R^m$. Then $S_{\infty}$ is connected. 
\end{thm}
\begin{proof}[Proof of Theorem \em\ref{property} \em for the case $n=2$]
Write $f:=(\frac{f_1}{f_0},\cdots,\frac{f_m}{f_0}):\R^2\to\R^m$ where each $f_i\in\R[\x_1,\x_2]$ and $f_0$ has an empty zero set. Keep notations from \ref{s3a} and assume $\gcd(f_0,f_1,\ldots,f_m)=1$, so the set of points of indeterminacy of $F_\K$ (where $\K=\R$ or $\C$) is
$$
Y_\K=\{x_0F_0'=0,F_1=0,\ldots,F_m=0\}\cap\K\PP^2.
$$
By \ref{res} $Y_\K$ is a finite subset of $\ell_\infty(\K)\cup(\{F_0'=0\}\cap\K\PP^2)$. As $f$ is quasi-polynomial, $Y_\R\subset\ell_\infty(\R)\setminus\{F_0'=0\}$. The proof is conducted in several steps:

\vspace{2mm}
\noindent{\bf Step 1}. \em Initial Preparation\em. Let $(Z_1,\pi_\C,\widehat{F}_\C)$ be an invariant resolution for $F_\C$. Keep notations from \ref{clue} and denote 
\begin{itemize}
\item[$\bullet$] $X_1:=Z_1\cap\R\PP^k$, which is a non-singular real projective surface.
\item[$\bullet$] $\pi_\R:=\pi_\C|_{X_1}:X_1\to\R\PP^2$, which is the composition of a sequence of finitely many blow-ups and its restriction $\pi_\R|_{X_1\setminus\pi_\R^{-1}(Y_\R)}:X_1\setminus\pi_\R^{-1}(Y_\R)\to\R\PP^2\setminus Y_\R$ is a biregular isomorphism.
\item[$\bullet$] $\widehat{F}_\R:=\widehat{F}_\C|_{X_1}:X_1\to\R\PP^m$, which is a real regular map and satisfies 
$$
\widehat{F}_\R|_{X_1\setminus\pi_\R^{-1}(Y_\R)}=F_\R\circ\pi_\R|_{X_1\setminus\pi_\R^{-1}(Y_\R)}.
$$
\end{itemize}
As $Y_\R\subset\ell_\infty(\R)$, it holds $\pi_\R(x)\in\R\PP^2\setminus\ell_\infty(\R)\equiv\R^2$ for each point $x\in X_1\setminus\pi_\R^{-1}(\ell_\infty(\R))$. Thus,
$$
\widehat{F}_\R(x)=F_\R\circ\pi_\R(x)=f(\pi_\R(x))\in\R^m\equiv\R\PP^m\setminus\mathsf{H}_\infty(\R),
$$ 
so $\widehat{F}_\R^{-1}(\mathsf{H}_\infty(\R))\subset\pi_\R^{-1}(\ell_\infty(\R))$.

\paragraph{}\label{comp}
The strict transform $\mathsf{K}_\infty$ under $\pi_\C$ of $\ell_\infty(\C)$ is an invariant non-singular rational curve and $\pi_\C|_{\mathsf{K}_\infty}:\mathsf{K}_\infty\to\ell_\infty(\C)\equiv\C\PP^1$ is an invariant biregular isomorphism. Consequently, $\mathsf{C}_\infty:=\mathsf{K}_\infty\cap\R\PP^k$ is a real non-singular rational curve and it is the strict transform under $\pi_\R$ of $\ell_\infty(\R)$.

\paragraph{}\label{E} Define $\mathsf{E}:=\{F_0'(z)=0\}\subset\C\PP^2$ and let $\widetilde{\mathsf{E}}$ be its strict transform under $\pi_\C$, which is an invariant projective curve. As $f_0$ has empty zero set and $\x_0$ does not divide $F_0'$, the intersection $\mathsf{E}\cap\R\PP^2$ is a finite subset of $\ell_\infty(\R)$. In addition $\pi_\C^{-1}(\mathsf{E})=\widetilde{\mathsf{E}}\cup\bigcup_{y\in\mathsf{E}\cap Y_\C}\pi_\C^{-1}(y)$. We have:

\paragraph{}\label{E1} $\widetilde{\mathsf{E}}\cap\bigcup_{y\in Y_\R}\pi^{-1}_\C(y)=\varnothing$ and $\widetilde{\mathsf{E}}\cap\R\PP^k\subset\mathsf{C}_\infty$.

Let $y\in Y_\R$ and let us show $\widetilde{\mathsf{E}}\cap\pi_\C^{-1}(y)=\varnothing$. Otherwise, there exists $x\in\widetilde{\mathsf{E}}\cap\pi_\C^{-1}(y)$, so $y=\pi_\C(x)\in\pi_\C(\widetilde{\mathsf{E}})=\mathsf{E}$. Consequently, $y\in Y_\R\cap\mathsf{E}=Y_\R\cap\{F_0'=0\}=\varnothing$, which is a contradiction. As $\pi_\C(\widetilde{\mathsf{E}}\cap\R\PP^k)\subset\mathsf{E}\cap\R\PP^2\subset\ell_\infty(\R)$, we have
$$
\widetilde{\mathsf{E}}\cap\R\PP^k\subset\pi_{\R}^{-1}(\ell_\infty(\R))\setminus\bigcup_{y\in Y_\R}\pi_\R^{-1}(y)\subset\mathsf{C}_\infty.
$$

\paragraph{}\label{E2} As $\widehat{F}_\C|_{Z_1\setminus\pi_\C^{-1}(Y_\C)}=F_\C\circ\pi_\C|_{Z_1\setminus\pi_\C^{-1}(Y_\C)}$, it holds
$$
\widetilde{\mathsf{E}}\cup \mathsf{K}_\infty\subset \widehat{F}_\C^{-1}(\mathsf{H}_\infty(\C))\subset\widetilde{\mathsf{E}}\cup \mathsf{K}_\infty\cup\pi_\C^{-1}(Y_\C).
$$
Thus, $\widehat{F}_\C^{-1}(\mathsf{H}_\infty(\C))$ is an invariant projective curve whose irreducible components different from $\widetilde{\mathsf{E}}$ are by \ref{clue} either singletons or non-singular rational curves.

\paragraph{}\label{diagram} The following diagram summarizes the achieved situation:

\begin{center}
\begin{tikzpicture}
 \node (e) at (0,3.9) {$\widetilde{\mathsf{E}}\cup\mathsf{K}_\infty$};
 \node (c) at (0,3.4) {$\cap$};
 \node (f) at (-0.2,2.9) {$\widehat{F}_\C^{-1}(\mathsf{H}_\infty(\C))$};
 \node (c2) at (0,2.4) {$\cap$};
	 \node (z1) at (0,1.95) {$Z_1$};
	 \node (cp2) at (3,2) {$\C\PP^2$};
	 \node (c2) at (3,2.4) {$\cap$};
	 \node (hr) at (3,2.9) {$\ell_\infty(\C)\cup\mathsf{E}$};
	 \node (ss) at (4.125,2.92) {$\supset$};
	 \node (hr) at (4.6,2.9) {$Y_\C$};
	 \node (cpm) at (3,0) {$\C\PP^m$};
	 \node (cup) at (3,-0.45) {$\cup$};
	 \node (hr) at (3,-1) {$\mathsf{H}_\infty(\C)$};
\draw [->] (z1.east) to node [above] {$\pi_\C$} (cp2.west);
\draw [->] (z1.south east) to node [left] {$\widehat{F}_\C$} (cpm.north west);
\draw [->] (cp2.south) to node [right] {$F_\C$} (cpm.north);
\end{tikzpicture}
\begin{tikzpicture}
 \node (e) at (0,3.9) {$\widetilde{\mathsf{E}}\cap\R\PP^k$};
 \node (c) at (0,3.4) {$\cap$};
 \node (f) at (0,2.9) {$\mathsf{C}_\infty$};
 \node (c2) at (0,2.4) {$\cap$};
	 \node (z1) at (0,1.95) {$X_1$};
	 \node (cp2) at (3,2) {$\R\PP^2$};
	 \node (c2) at (3,2.4) {$\cap$};
	 \node (hr) at (3,2.9) {$\ell_\infty(\R)$};
	 \node (ss) at (3.75,2.92) {$\supset$};
	 \node (hr) at (4.25,2.9) {$Y_\R$};
	 \node (cpm) at (3,0) {$\R\PP^m$};
	 \node (cup) at (3,-0.45) {$\cup$};
	 \node (hr) at (3,-1) {$\mathsf{H}_\infty(\R)$};
\draw [->] (z1.east) to node [above] {$\pi_\R$} (cp2.west);
\draw [->] (z1.south east) to node [left] {$\widehat{F}_\R$} (cpm.north west);
\draw [->] (cp2.south) to node [right] {$F_\R$} (cpm.north);
\end{tikzpicture}
\end{center}

\noindent{\bf Step 2.} We prove next: 

\paragraph{}\label{usrc}\em $\widehat{F}_\C^{-1}(\mathsf{H}_\infty(\C))$ is connected\em. Consequently, \em $\widehat{F}_\C^{-1}(\mathsf{H}_\infty(\C))$ does not contain isolated points and its irreducible components different from $\widetilde{\mathsf{E}}$ are non-singular rational curves\em.

Indeed, by Stein's factorization theorem \cite[III.11.5]{h2} applied to the projective morphism $\widehat{F}_\C:Z_1\to\C\PP^m$, there exist a projective variety $V$ and projective morphisms $G_1:Z_1\to V$ and $G_2:V\to\C\PP^m$ such that: 
\begin{itemize}
\item[$\bullet$]$G_1$ is surjective and its fibers are connected, 
\item[$\bullet$]$G_2$ is a finite morphism and
\item[$\bullet$]$\widehat{F}_\C=G_2\circ G_1$. 
\end{itemize}
To prove \ref{usrc} it is enough to show that $\mathsf{H}:=G_2^{-1}(\mathsf{H}_\infty(\C))$ and $G_1^{-1}(\mathsf{H})=\widehat{F}_\C^{-1}(\mathsf{H}_\infty(\C))$ are connected.

\paragraph{} \em $\mathsf{H}$ is connected\em. 

As $G_2$ is finite, it is by \cite[II.5.17]{h2} an affine morphism. Thus,
$$
G_2^{-1}(\C^m)=G_2^{-1}(\C\PP^m)\setminus G_2^{-1}(\mathsf{H}_\infty(\C))=V\setminus\mathsf{H}
$$ 
is an affine algebraic variety. As $V$ is a complete manifold, we deduce by \cite[6.2, p.79]{h} that $\mathsf{H}=V\setminus G_2^{-1}(\C^m)$ is connected because it is the complement in $V$ of an affine open subvariety.

\paragraph{} \em $G_1^{-1}(\mathsf{H})$ is connected\em.

Suppose that $G_1^{-1}(\mathsf{H})$ is the disjoint union of two closed subsets $A_1,A_2$. As $G_1$ is proper and surjective, $G_1(A_1),G_1(A_2)$ are closed subsets of $\mathsf{H}$ and $\mathsf{H}=G_1(A_1)\cup G_1(A_2)$. In case $G_1(A_i)\neq\varnothing$ for $i=1,2$, the intersection $G_1(A_1)\cap G_1(A_2)\neq\varnothing$ because $\mathsf{H}$ is connected. If $x\in G_1(A_1)\cap G_1(A_2)$, the fiber 
$$
G_1^{-1}(\{x\})=(G_1^{-1}(\{x\})\cap A_1)\sqcup(G_1^{-1}(\{x\})\cap A_2)
$$ 
is the disjoint union of two non-empty closed sets, so $G_1^{-1}(\{x\})$ is disconnected, which is a contradiction because the fibers of $G_1$ are connected.

\vspace{2mm}
\noindent{\bf Step 3.} In the following we use Lemma \ref{crack} several times. To ease the procedure we point out the key facts. By \ref{curve2} $A:=\pi_\C^{-1}(\ell_\infty(\C)\cup Y_\C)$ is an algebraic curve whose irreducible components are non-singular rational curves. Observe that $U:=Z_1\setminus A=\pi_\C^{-1}(\C^2\setminus Y_\C)$ is a dense subset of $Z_1$ biregularly equivalent to $\C^2\setminus Y_\C$ (that is, the complement in $\C^2$ of a finite subset). 

\paragraph{}\label{ucrack} \em Let $B_1,B_2\subset A$ be two connected compact algebraic curves without common irreducible components. Then the intersection $B_1\cap B_2$ is \em by Lemma \ref{crack}, Example \ref{bouquet} and Remark \ref{crackc} \em either empty or a singleton\em.

\vspace{2mm}
\noindent{\bf Step 4.} By Zariski's Main Theorem \cite[III.11.4]{h2} applied to the birational projective morphism $\pi_\C:Z_1\to\C\PP^2$, \em the fiber $\pi_\C^{-1}(y)$ is connected for each $y\in Y_\C$\em. We prove next: 

\paragraph{}\label{connected2}\em If $y\in Y_\R$, the invariant projective variety $T_y:=\pi_\C^{-1}(y)\cap \widehat{F}_\C^{-1}(\mathsf{H}_\infty(\C))$ is connected and $\mathsf{K}_\infty\cap T_y$ is a singleton\em. In addition, \em if $T_y$ has dimension $1$ and $\mathsf{K}_{1,y},\ldots,\mathsf{K}_{r_y,y}$ are the invariant irreducible components of $T_y$, the projective curve $\bigcup_{i=1}^{r_y}\mathsf{K}_{i,y}$ is connected and 
$$
\mathsf{K}_\infty\cap T_y=\mathsf{K}_\infty\cap\bigcup_{i=1}^{r_y}\mathsf{K}_{i,y}
$$ 
is a singleton contained in $X_1$\em.

The clue to prove \ref{connected2} is the following:

\paragraph{}\label{connected0}\em Let $y\in Y_\R$ and $T$ be an invariant connected union of irreducible components of $\pi^{-1}_\C(y)$. Let $K_1,\ldots,K_r$ be the invariant irreducible components of $T$ and denote $C_i:=K_i\cap\R\PP^k$, which is \em by \ref{clue}(v) \em a real non-singular rational curve for $i=1,\ldots,r$. We have:
\begin{itemize}
\item[(i)] The projective curve $\bigcup_{i=1}^rK_i$ is connected.
\item[(ii)] Suppose moreover $\mathsf{K}_\infty\cap T\neq\varnothing$. Then the intersection $\mathsf{K}_\infty\cap T=\mathsf{K}_\infty\cap\bigcup_{i=1}^rK_i$ is a singleton contained in $X_1$. In particular, $\mathsf{K}_\infty\cup\bigcup_{i=1}^rK_i$ is connected.
\item[(iii)] We may order the indices $i=1,\ldots,r$ in such a way that $\mathsf{C}_\infty\cap C_1=\mathsf{K}_\infty\cap T$ is a singleton and $C_i\cap\bigcup_{j=1}^{i-1}C_j$ is a singleton for $i=2,\ldots,r$. In particular, the real projective curve $C:=\bigcup_{i=1}^rC_i$ is connected and $\mathsf{C}_\infty\cap C\neq\varnothing$.
\end{itemize}\em

We prove first \ref{connected0}(i). If $T=\bigcup_{i=1}^rK_i$, there is nothing to prove. 

Otherwise, denote the non-invariant irreducible components of $T$ with $K_{r+1},\ldots,K_s$. Denote
$$
t:=\max\Big\{\#{\mathcal F}:\ {\mathcal F}\subset\{1,\ldots,r\},\ \bigcup_{i\in{\mathcal F}}K_i\quad\text{is connected}\Big\}
$$
and let us check $t=r$. Suppose by contradiction $t<r$. We may assume that $K:=\bigcup_{i=1}^tK_i$ is connected. As each $K_\ell$ is connected, each intersection $K\cap K_\ell=\varnothing$ for $t<\ell\leq r$. As $T$ is connected and invariant and $K\cap\bigcup_{i=t+1}^rK_i=\varnothing$, we may assume $K\cap K_{r+1}\neq\varnothing$ and $\sigma(K_{r+1})=K_{r+2}$. As $K$ is invariant, $K\cap K_{r+2}\neq\varnothing$, so $K\cup K_{r+1}\cup K_{r+2}$ is connected and invariant. Repeating the previous argument recursively, we find indices $r<r+2j_0\leq s$ and $t<\ell\leq r$ such that (after reordering the indices $i=r+1,\ldots,s$) $K_\ell\cap(K\cup\bigcup_{i=r+1}^{r+{2j_0}}K_j)\neq\varnothing$ and for each $1\leq j\leq j_0$ it holds
\begin{itemize}
\item[$\bullet$]$K\cup\bigcup_{i=r+1}^{r+{2j}}K_i$ is connected and invariant and
\item[$\bullet$]$\sigma(K_{r+2j-1})=K_{r+2j}$.
\end{itemize} 
Such indices $r+2j_0,\ell$ exist because $T=\bigcup_{i=1}^sK_i$ is connected.

The invariant connected projective curves $K_\ell$ and $K\cup\bigcup_{i=r+1}^{r+{2j}}K_j$ are contained in $A:=\pi_\C^{-1}(\ell_\infty(\C)\cup Y_\C)$ because $y\in Y_\R$. The non-empty invariant intersection $K_\ell\cap(K\cup\bigcup_{i=r+1}^{r+{2j}}K_j)$ is by \ref{ucrack} a singleton $\{p_\ell\}\subset X_1=Z_1\cap\R\PP^k$. As $K_i\subset\C\PP^k\setminus\R\PP^k$ for $i=r+1,\ldots,s$ (see \ref{clue}(iv)), we have $p_\ell\in K_\ell\cap K\neq\varnothing$, which is a contradiction. Then $r=t$, so $\bigcup_{i=1}^rK_i$ is connected.

Next we prove \ref{connected0}(ii). Since $\mathsf{K}_\infty,T\subset A$ are invariant connected projective curves, the non-empty invariant intersection $\mathsf{K}_\infty\cap T$ is by \ref{ucrack} a singleton $\{p\}\subset X_1$. Thus, $p\in T\cap\R\PP^k\subset\bigcup_{i=1}^rK_i$ because the non-invariant irreducible components of $T$ are by \ref{clue}(v) contained in $\C\PP^k\setminus\R\PP^k$. Consequently, $\mathsf{K}_\infty\cap T=\mathsf{K}_\infty\cap\bigcup_{i=1}^rK_i=\{p\}$.

Finally, we show \ref{connected0}(iii). As $p\in X_1=Z_1\cap\R\PP^k$, 
$$
\mathsf{C}_\infty\cap C=\mathsf{C}_\infty\cap\bigcup_{i=1}^rC_i=\Big(\mathsf{K}_\infty\cap\bigcup_{i=1}^rK_i\Big)\cap\R\PP^k=\{p\}\neq\varnothing.
$$
We may assume $\mathsf{C}_\infty\cap C_1\neq\varnothing$, that is, $\mathsf{C}_\infty\cap C_1=\{p\}=\mathsf{K}_\infty\cap T$. As $\bigcup_{i=1}^rK_r$ is connected, we claim: \em We may order the indices $i=2,\ldots,r$ in such a way that $K_i$ intersects the union $\bigcup_{j=1}^{i-1}K_j$ for $i=2,\ldots,r$\em. 

Indeed, as $K=\bigcup_{i=1}^rK_i$ is connected, the intersection $K_1\cap\bigcup_{i=2}^rK_i\neq\varnothing$ and we assume $K_1\cap K_2\neq\varnothing$. As $K$ is connected, the intersection of $K_1\cup K_2$ and $\bigcup_{i=3}^rK_i\neq\varnothing$ is non-empty. We may assume that $K_3$ intersects $K_1\cup K_2$ and proceeding this way we prove the claim.

Next, since $K_i$ and $\bigcup_{j=1}^{i-1}K_j$ are invariant connected projective curves contained in $A$, the non-empty invariant intersection $K_i\cap\bigcup_{j=1}^{i-1}K_j$ is by \ref{ucrack} a singleton $\{q_i\}\subset X_1$. Thus, 
$$
q_i\in\R\PP^k\cap K_i\cap\bigcup_{j=1}^{i-1}K_j=C_i\cap\bigcup_{j=1}^{i-1}C_j,
$$
so $C_i\cap\bigcup_{j=1}^{i-1}C_j\neq\varnothing$ for $i=2,\ldots,r$. As each $C_i$ is a non-singular curve biregularly equivalent to $\R\PP^1$, we deduce that the projective curve $C:=\bigcup_{i=1}^rC_i$ is connected.

\paragraph{}\label{connected2p} Now we are ready to prove \ref{connected2}. As $y\in Y_\R\subset\ell_{\infty}(\R)\subset\ell_{\infty}(\C)$ and $\pi_\C(\mathsf{K}_\infty)=\ell_{\infty}(\C)$, we deduce $\pi_\C^{-1}(y)\cap\mathsf{K}_\infty\neq\varnothing$. If $y_1\neq y_2$, the intersection $\pi_\C^{-1}(y_1)\cap\pi_\C^{-1}(y_2)=\varnothing$ and by \ref{E1} $\widetilde{\mathsf{E}}\cap\pi_\C^{-1}(y)=\varnothing$. Consequently, by \ref{E2}
\begin{multline*}
T_y\cap\mathsf{K}_\infty=\pi_\C^{-1}(y)\cap \widehat{F}_\C^{-1}(\mathsf{H}_\infty(\C))\cap\mathsf{K}_\infty\\
=((\pi^{-1}_\C(y)\cap\mathsf{K}_\infty)\cup\pi^{-1}_\C(y))\cap\mathsf{K}_\infty=\pi^{-1}_\C(y)\cap\mathsf{K}_\infty\neq\varnothing,
\end{multline*} 
so by \ref{connected0}(ii) $T_y\cap\mathsf{K}_\infty$ is a singleton. Denote 
$$
R_y:=\widetilde{\mathsf{E}}\cup\mathsf{K}_\infty\cup\bigsqcup_{z\in Y_\C,z\neq y}(\pi^{-1}_\C(z)\cap\widehat{F}_\C^{-1}(\mathsf{H}_\infty(\C)))
$$
and observe $\widehat{F}_\C^{-1}(\mathsf{H}_\infty(\C))=T_y\cup R_y$ by \ref{E2}. Using again $\widetilde{\mathsf{E}}\cap\pi_\C^{-1}(y)=\varnothing$ and $\pi_\C^{-1}(y_1)\cap\pi_\C^{-1}(y_2)=\varnothing$ if $y_1\neq y_2$, we deduce $T_y\cap R_y=T_y\cap\mathsf{K}_\infty$, which is a singleton. Consequently, if $T_y$ was disconnected, then $\widehat{F}_\C^{-1}(\mathsf{H}_\infty(\C))$ would have been disconnected, which contradicts \ref{usrc}. Thus, the first part of claim \ref{connected2} holds. The second part follows readily from \ref{connected0}.

\vspace{2mm}
\noindent{\bf Step 5.} Next we show:

\paragraph{}\label{connected}\em $\widehat{F}_\R^{-1}(\mathsf{H}_\infty(\R))$ is connected\em. Consequently, \em $\widehat{F}_\R(\widehat{F}_\R^{-1}(\mathsf{H}_\infty(\R)))$ is also connected\em.

As $\widehat{F}_\C$ is invariant, $\widehat{F}_\R^{-1}(\mathsf{H}_\infty(\R))=\widehat{F}_\C^{-1}(\mathsf{H}_\infty(\C))\cap\R\PP^k$. By \ref{E2} 
\begin{equation}\label{star}
\widehat{F}_\C^{-1}(\mathsf{H}_\infty(\C))=\widetilde{\mathsf{E}}\cup\mathsf{K}_\infty\cup\bigcup_{y\in Y_\C}(\pi^{-1}_\C(y)\cap \widehat{F}_\C^{-1}(\mathsf{H}_\infty(\C))).
\end{equation} 
\indent Fix $y\in Y_\R$ and consider $T_y:=\pi_\C^{-1}(y)\cap \widehat{F}_\C^{-1}(\mathsf{H}_\infty(\C))$. If $T_y$ is a singleton, we deduce by \ref{connected2} that $T_y\subset\mathsf{K}_\infty$. Otherwise we denote the invariant irreducible components of $T_y$ with $\mathsf{K}_{1,y},\ldots,\mathsf{K}_{r_y,y}$. As the non-invariant irreducible components of $\pi_\C^{-1}(y)$ do not intersect $\R\PP^k$ by \ref{clue}(v), we deduce $T_y\cap\R\PP^k=\bigcup_{i=1}^{r_y}\mathsf{K}_{i,y}\cap\R\PP^k$. 

In addition by \ref{clue}(iv) $\pi_\C^{-1}(y)\cap\R\PP^k=\varnothing$ for all $y\in Y_\C\setminus Y_\R$. Denote the subset of points of $Y_\R$ such that $T_y$ is not a singleton with $Y_\R'$. If we intersect expression \eqref{star} with $\R\PP^k$, we deduce by \ref{E1}
$$
\widehat{F}_\R^{-1}(\mathsf{H}_\infty(\R))=\widehat{F}_\C^{-1}(\mathsf{H}_\infty(\C))\cap\R\PP^k=\mathsf{C}_\infty\cup\bigcup_{y\in Y_\R'}\bigcup_{i=1}^{r_y}\mathsf{K}_{i,y}\cap\R\PP^k.
$$
\indent By \ref{connected2} the projective curve $\bigcup_{i=1}^{r_y}\mathsf{K}_{i,y}$ is connected and the intersection $\mathsf{K}_\infty\cap\bigcup_{i=1}^{r_y}\mathsf{K}_{i,y}$ is a singleton $\{p_y\}$ for each $y\in Y_\R'$. By \ref{connected0}(iii) each projective curve $\mathsf{C}_y:=\bigcup_{i=1}^{r_y}\mathsf{C}_{i,y}$ is connected and each intersection $\mathsf{C}_\infty\cap\mathsf{C}_y\neq\varnothing$. As $\mathsf{C}_\infty=\mathsf{K}_\infty\cap\R\PP^k$ is by \ref{comp} connected, we conclude that
$$
\widehat{F}_\R^{-1}(\mathsf{H}_\infty(\R))=\mathsf{C}_\infty\cup\bigcup_{y\in Y_\R'}\bigcup_{i=1}^{r_y}\mathsf{C}_{i,y}
$$
is connected too.

\vspace{2mm}
\noindent{\bf Final Step.} {\em Conclusion.} As $\widehat{F}_\R(\widehat{F}_\R^{-1}(\mathsf{H}_\infty(\R)))$ is connected, to prove that $S_\infty$ is connected, too, it is enough to show that both sets are equal, that is,
$$
\widehat{F}_\R(\widehat{F}_\R^{-1}(\mathsf{H}_\infty(\R)))=S_\infty. 
$$
As $X_1=Z_1\cap\R\PP^k$ is compact, $\widehat{F}_\R$ is proper, so the restriction 
$$
\widehat{F}_\R|_{\widehat{F}_\R^{-1}(\R^m)}:\widehat{F}_\R^{-1}(\R^m)\to\R^m
$$ 
is also proper. As $\pi^{-1}_\R(\R^2)$ is dense in $X_1$ and $S=f(\R^2)=F_\R(\R^2)=\widehat{F}_\R(\pi_\R^{-1}(\R^2))$, we have
\begin{multline*}
\widehat{F}_\R(\widehat{F}_\R^{-1}(\R^m))=\widehat{F}_\R(\cl_{\widehat{F}_\R^{-1}(\R^m)}(\pi_\R^{-1}(\R^2)))\\
=\cl_{\R^m}(\widehat{F}_\R(\pi_\R^{-1}(\R^2)))=\cl_{\R^m}(f(\R^2))=\cl_{\R^m}(S),
\end{multline*}
so $\widehat{F}_\R(\widehat{F}_\R^{-1}(\R^m))=\cl_{\R^m}(S)$. As $\widehat{F}_\R$ is proper,
\begin{multline*}
\cl_{\R^m}(S)\sqcup S_\infty=\cl_{\R\PP^m}(S)=\cl_{\R\PP^m}(\widehat{F}_\R(\widehat{F}_\R^{-1}(\R^m)))\\
=\widehat{F}_\R(\cl_{X_1}(\widehat{F}_\R^{-1}(\R^m)))=\widehat{F}_\R(X_1)=\widehat{F}_\R(\widehat{F}_\R^{-1}(\R^m))\sqcup \widehat{F}_\R(\widehat{F}_\R^{-1}(\mathsf{H}_\infty(\R))).
\end{multline*}
Consequenlty, $S_\infty=\widehat{F}_\R(\widehat{F}_\R^{-1}(\mathsf{H}_\infty(\R)))$, which is by \ref{connected} connected, as required.
\end{proof}

We are ready to prove Theorem \ref{main1}. We present an independent proof from the one of Theorem \ref{property}. This is enlightening for the proof of Theorem \ref{property} for $n\geq3$.

\begin{proof}[Proof of Theorem \em\ref{main1}]
For $n=1$ the result follows from \cite[1.1]{fe}. To prove that $S_{\infty}$ is connected if $n\geq2$, it is enough to show that for any given pair of points $p,q\in S_{\infty} $ there exists a connected subset of $S_\infty$ containing $p$ and $q$. By Lemma \ref{polcurve} there exist polynomials ${\tt p}_i,{\tt q}_i\in\R[\t]$ such that ${\tt p}_i(0),{\tt q}_i(0)\neq0$ and integers $k_i,\ell_i$ such that the rational paths $\alpha:=(\t^{k_1}{\tt p}_1,\ldots,\t^{k_n}{\tt p}_n)$ and $\beta:=(\t^{\ell_1}{\tt q}_1,\ldots,\t^{\ell_n}{\tt q}_n)$ satisfy
$$
\lim_{t\to0^+}(f\circ\alpha)(t)=p\quad\text{and}\quad\lim_{t\to0^+}(f\circ\beta)(t)=q.
$$
At least one couple $(k_i,\ell_j)$ is of negative integers. Consider the polynomials
$$
P_i(\x,\y):=\left\{\begin{array}{cl}
\y^{|k_i|}{\tt p}_i(\x)&\text{if $k_i<0$,}\\[4pt]
\x^{k_i}{\tt p}_i(\x)&\text{if $k_i\geq0$}
\end{array}\right.
\quad\text{and}\quad
Q_i(\x,\y):=\left\{\begin{array}{cl}
(-\y)^{|\ell_i|}{\tt q}_i(\x)&\text{if $\ell_i<0$,}\\[4pt]
\x^{\ell_i}{\tt q}_i(\x)&\text{if $\ell_i\geq0$}
\end{array}\right.
$$
and let 
$$
h:=\frac{\x\y+1}{2}(P_1,\ldots,P_n)+\frac{1-\x\y}{2}(Q_1,\ldots,Q_n).
$$
Consider the polynomial map $g:=f\circ h:\R^2\to\R^m$ and observe
\begin{itemize}
\item[$\bullet$]$T_0:=\im(g)\subset\im(f)=S$, so $T_{0,\infty}\subset S_{\infty}$.
\item[$\bullet$]$p=\lim_{t\to0^+}g(t,\frac{1}{t})\in T_{0,\infty}$ and $q=\lim_{t\to0^+}g(t,-\frac{1}{t})\in T_{0,\infty}$.
\end{itemize}
As $T_{0,\infty}$ is connected for $n=2$ by Theorem \ref{property}, we are done. 
\end{proof}

\subsection{Proof of Theorem \ref{property}}\setcounter{paragraph}{0} 
Now we prove Theorem \ref{property} for an arbitrary $n$.

\begin{proof}[Proof of Theorem \em \ref{property} \em for an arbitrary $n$.]
The case $n=1$ follows from \cite[1.4]{fe}. Assume $n\geq2$ and write $f:=(\frac{f_1}{f_0},\ldots,\frac{f_m}{f_0})$ where each $f_j\in\R[\x]$, $\gcd(f_0,f_1,\ldots,f_m)=1$ and $f_0$ does not vanish on $\R^n$. The proof is conducted in several steps:

\paragraph{} \em Initial assumptions to simplify the proof\em.
Assume $\deg(f_1)\geq\deg(f_j)$ for $j=1,\ldots,m$. After a change of the type $(y_1,\ldots,y_m)\mapsto(y_1,y_2+b_2y_1,\ldots,y_m+b_my_1)$ where $b_j\in\R$ we can suppose 
$$
\deg(f_1)=\cdots=\deg(f_m)=d>\deg(f_0)=d-e
$$ 
for some $e\geq1$. Denote 
$$
F_j:=\x_0^df_j\Big(\frac{\x_1}{\x_0},\ldots,\frac{\x_n}{\x_0}\Big)
$$ 
and $F_0=\x_0^eF_0'$ where $e\geq1$ and $F_0'\in\R[\x_0,\x_1,\ldots,\x_n]$ is not divisible by $\x_0$. Notice that $\x_0$ does not divide $F_j$ for $j=1,\ldots,m$ because $\deg(f_j)=\deg(F_j)=d$. After a change of the type $(x_1,\ldots,x_m)\mapsto(x_1,x_2+a_2x_1,\ldots,x_n+a_nx_1)$ where $a_i\in\R$ we can suppose $\deg(f_j)=\deg_{\x_1}(f_j)$ for each $j$. 

\paragraph{} \em Equivalent formulation for the statement\em.
To prove that $S_{\infty}$ is connected, it is equivalent to show: \em There exists a point $p_0\in S_{\infty}$ such that for any other point $q\in S_{\infty}$ there exists a connected subset of $S_\infty$ containing $p_0$ and $q$\em. 

We will prove this last fact. Fix 
$$
p_0:=\lim_{t\to0^+}(F_0:F_1:\cdots:F_m)\Big(1:\frac{1}{t}:0:\cdots:0\Big)\in\mathsf{H}_\infty(\R)
$$ 
and let $q\in S_{\infty} $. By Lemma \ref{polcurve} there exist a rational path $$
\alpha:=(\t^{k_1}{\tt p}_1,\ldots,\t^{k_n}{\tt p}_n)\in\R(\t)^n
$$ 
where
\begin{itemize}
\item[$\bullet$]$k_i\in\Z$, $k_{i_0}:=\min\{k_1,\ldots,k_n\}<0$, 
\item[$\bullet$]${\tt p}_i\in\R[\t]$ with ${\tt p}_i(0)\neq 0$ for $i=1,\ldots,n$ and ${\tt p}_{i_0}=\pm1$
\end{itemize} 
and an integer $r\geq1$ such that $q=\lim_{t\to0^+}(f\circ(\alpha+\beta))(t)$ for each $\beta\in(\t)^r\R[\t]^n$. After the change $\t\to\t^6$ we may assume $k_{i_0}\leq-6$ and even. We keep all initial notations.

\paragraph{} \em Construction of an auxiliary regular map\em.
Write ${\tt p}_i:=\sum_{j=0}^{d_i}a_{ij}\t^j$ where $d_i:=\deg({\tt p}_i)$ and consider the formula
$$
P_i(\x,\y):=\sum_{j+k_i<0}a_{ij}\y^{|k_i+j|}+\sum_{j+k_i\geq0}a_{ij}\frac{\x^{e_j}\y}{((\x\y-1)^2+\y^4)^{q_j}}
$$
where $j+k_i+1=4q_j+e_j$, $q_j\geq0$ and $0\leq e_j\leq3$ for each $0\leq j\leq d_i$ such that $j+k_i\geq0$. Observe $P_i(\t,\frac{1}{\t})=\t^{k_i}{\tt p}_i(\t)$ for $i=1,\ldots,n$ and $P_{i_0}={\rm p}_{i_0}\y^{|k_{i_0}|}=\pm\y^{|k_{i_0}|}$. 

Denote $\ell_0:=\max_i\{q_{d_i}:\ d_i+k_i\geq0\}$ and notice $((\x\y-1)^2+\y^4)^\ell P_i\in\R[\x,\y]$ for each $\ell\geq\ell_0$. In addition
\begin{multline*}
\deg(((\x\y-1)^2+\y^4)^\ell P_i)\leq\max\{-k_i+4\ell,4\ell-4q_j+e_j+1:\ k_i+j\geq0\}\\
\leq 4\ell+\max\{-k_i,4\}\leq4\ell+|k_{i_0}|
\end{multline*}
and the equality $\deg(((\x\y-1)^2+\y^4)^\ell P_i)=4\ell+|k_{i_0}|$ holds if and only if $k_i=k_{i_0}$. As $|k_{i_0}|\geq6$, 
$$
\deg_\x(((\x\y-1)^2+\y^4)^\ell P_i)\leq
\left\{\!\!\begin{array}{ll}
2(\ell-q_{d_i})+3&\text{if $k_i+d_i\geq0$}\\[4pt]
2\ell&\text{if $k_i+d_i<0$}
\end{array}\!\!\right\}
< 2\ell+4<4\ell+|k_{i_0}|
$$ 
for all $i=1,\ldots,n$. Let $\ell:=\max\{\ell_0,r\}$ and define
$$
h_i(\x,\y):=\left\{\begin{array}{ll}
((\x\y-1)^2+\y^4)^\ell&\text{if $i=0$,}\\[4pt]
((\x\y-1)^2+\y^4)^\ell P_1(\x,\y)+{\rm p}_1(0)\x^{4\ell+|k_{i_0}|}&\text{if $i=1$,}\\[4pt]
((\x\y-1)^2+\y^4)^\ell P_i(\x,\y)&\text{if $i=2,\ldots,n$.}
\end{array}\right.
$$
Consider the regular map $h:=(\frac{h_1}{h_0},\ldots,\frac{h_n}{h_0}):\R^2\to\R^n$ and observe
\begin{multline*}
h(\t,0)=({\rm p}_1(0)\t^{4\ell+|k_{i_0}|},0,\ldots,0)\quad
\text{and}\quad h\Big(\t,\frac{1}{\t}\Big)=\alpha(\t)+({\rm p}_1(0)\t^{8\ell+|k_{i_0}|},0,\ldots,0),
\end{multline*}
so $\lim_{t\to0^+}(f\circ h)(\frac{1}{t},0)=p_0$ and $\lim_{t\to0^+}(f\circ h)(t,\frac{1}{t})=q$.

\paragraph{} \em Construction of an auxiliary quasi-polynomial map such that $p_0,q$ belong to the set of points at infinity of its image\em.
Let $g:=f\circ h:\R^2\to\R^m$ and denote $g_i:=F_i(h_0,h_1,\ldots,h_n)\in\R[\u_1,\u_2]$ for $i=0,1,\ldots,m$. Observe $g=(\frac{g_1}{g_0},\ldots,\frac{g_m}{g_0})$ and $g_0$ does not vanish on $\R^2$. We claim: \em $g$ is a quasi-polynomial map and $p_0,q\in g(\R^2)_{\infty}$\em. 

First we have
\begin{itemize}
\item[$\bullet$] $T_0:=\im(g)\subset\im(f)=S$, so $T_{0,\infty}\subset S_{\infty}$.
\item[$\bullet$] $p_0=\lim_{t\to0^+}g(\frac{1}{t},0)\in T_{0,\infty}$ and $q=\lim_{t\to0^+}g(t,\frac{1}{t})\in T_{0,\infty}$.
\item[$\bullet$] $g_0=((\u_1\u_2-1)^2+\u_2^4)^{e\ell}F_0'(h_0,h_1,\ldots,h_n)$.
\end{itemize}
\noindent In addition, 
$$
\deg(g_0)=d(4\ell+|k_{i_0}|)-e|k_{i_0}|<d(4\ell+|k_{i_0}|)=\max\{\deg(g_1),\ldots,\deg(g_m)\}.
$$ 

Indeed, since 
\begin{itemize}
\item[$\bullet$] $\deg(F_0')=\deg(f_0)=\deg_{\x_1}(f_0)=d-e$,
\item[$\bullet$] $\deg(F_j)=\deg(f_j)=\deg_{\x_1}(f_j)=d$,
\item[$\bullet$] $\deg(h_i)\leq 4\ell+|k_{i_0}|,\ \deg(h_1)=4\ell+|k_{i_0}|$ and
\item[$\bullet$] $(h_0,h_1,\ldots,h_n)(\t,0)=(1,{\rm p}_1(0)\t^{4\ell+|k_{i_0}|},0,\ldots,0)$, 
\end{itemize} 
we have
$$
\deg(g_0)=4e\ell+(d-e)(4\ell+|k_{i_0}|)=d(4\ell+|k_{i_0}|)-e|k_{i_0}|<d(4\ell+|k_{i_0}|)=\deg(g_j)
$$ 
for each $j=1,\ldots,m$.

Let $\mu:=4\ell+|k_{i_0}|$ and write $H_i:=\u_0^\mu h_i(\frac{\u_1}{\u_0},\frac{\u_2}{\u_0})$ and $G_j:=\u_0^{d\mu}g_j(\frac{\u_1}{\u_0},\frac{\u_2}{\u_0})$, which are homogeneous polynomials. Notice that
\begin{multline*}
G_j=\u_0^{d\mu}F_j\Big(h_0\Big(\frac{\u_1}{\u_0},\frac{\u_2}{\u_0}\Big),h_1\Big(\frac{\u_1}{\u_0},\frac{\u_2}{\u_0}\Big),\ldots,h_n\Big(\frac{\u_1}{\u_0},\frac{\u_2}{\u_0}\Big)\Big)=F_j(H_0,H_1,\ldots,H_n)\\[-15pt]
\end{multline*}
and $G_0=\u_0^{e|k_{i_0}|}G_0'$ where $G_0':=((\u_1\u_2-\u_0^2)^2+\u_2^4)^{e\ell}F_0'(H_0,H_1,\ldots,H_n)$. Counting degrees one realizes: \em $\u_0$ does not divide $G_0'$\em. 

In the following all zero sets are considered in $\R\PP^n$. To prove that $g$ is quasi-polynomial it only remains to check $\{G_0'=0,G_1=0,\ldots,G_m=0\}=\varnothing$. 

Indeed, as $\{F_0'=0,F_1=0,\ldots,F_n=0\}=\varnothing$, the following equality holds
\begin{multline*}
\{G_0'=0,G_1=0,\ldots,G_m=0\}=\{H_0=0,\ldots,H_n=0\}\\
\cup\{((\u_1\u_2-\u_0^2)^2+\u_2^4)^{e\ell}=0,G_1=0,\ldots,G_m=0\}.
\end{multline*}
Notice
\begin{align*}
H_0(\u_0,\u_1,\u_2)&=u_0^{|k_{i_0}|}((\u_1\u_2-\u_0^2)^2+\u_2^4)^\ell,\\
H_1(0,\u_1,\u_2)&=\begin{cases}
{\rm p_1}(0)\u_1^{4\ell+|k_{i_0}|}&\text{if $k_{i0}<k_1$},\\[4pt]
{\rm p_1}(0)\left((\u_1^2\u_2^2+\u_2^4)^\ell\u_2^{|k_{i_0}|}+\u_1^{4\ell+|k_{i_0}|}\right)&\text{if $k_{i0}=k_1$},
\end{cases}\\
H_i(0,0,\u_2)&=\begin{cases}
0&\text{if $i\geq2$ and $k_{i0}<k_i$},\\[4pt]
{\rm p_i}(0)\u_2^{4\ell+|k_{i_0}|}&\text{if $i\geq 2$ and $k_{i0}=k_i$}.
\end{cases}
\end{align*}
Consequently, $\{H_0=0,\ldots,H_n=0\}=\varnothing$ because
\begin{itemize}
\item[$\bullet$] $H_0=0$ provides $\u_0=0$, 
\item[$\bullet$] $H_1(0,\u_1,\u_2)=0$ provides $\u_1=0$ (we have used here that $k_{i_0}$ is even) and
\item[$\bullet$] $H_{i_0}(0,0,\u_2)=0$ provides $\u_2=0$.
\end{itemize}
On the other hand $((\u_1\u_2-\u_0^2)^2+\u_2^4)^{e\ell}=0$ provides $\u_0=0,\u_2=0$. As $G_1=0$, we have 
\begin{multline*}
0=G_1(0,\u_1,0)\\
=F_1(H_0(0,\u_1,0),H_1(0,\u_1,0),\ldots,H_n(0,\u_1,0))=F_1(0,\u_1^\mu,0,\ldots,0).
\end{multline*}
For the last equality use that $\deg_\x(((\x\y-1)^2+\y^4)^\ell P_i)<4\ell+|k_{i_0}|$ for all $i=1,\ldots,n$. As $\deg(F_1)=\deg(f_1)=\deg_{\x_1}(f_1)=\deg_{\x_1}(F_1)$, we obtain $a:=F_1(0,1,0,\ldots,0)\neq0$. As
$$
0=F_1(0,\u_1^\mu,0,\ldots,0)=a\u_1^{d\mu},
$$
we get $\u_1=0$, so $\{((\u_1\u_2-\u_0^2)^2+\u_2^4)^{e\ell}=0,G_1=0,\ldots,G_m=0\}=\varnothing$. Therefore 
$$
\{G_0'=0,G_1=0,\ldots,G_m=0\}=\varnothing,
$$ 
so $g:\R^2\to\R^m$ is a quasi-polynomial map.

\paragraph{}\em Conclusion\em.
By Theorem \ref{property} for the case $n=2$ (already proved in \ref{s3a}) applied to the quasi-polynomial map $g$, we deduce that $T_{0,\infty}=(g(\R^2))_\infty\subset S_\infty$ is connected and since $p_0,q\in T_{0,\infty}$, we are done.
\end{proof}

The set of points at infinity of a semialgebraic set $S\subset\R^m$ is a semialgebraic subset of the hyperplane of infinity $\mathsf{H}_\infty(\R)$ of $\R\PP^m$. It seems reasonable to ask the following. 

\begin{squestion}\em
Let $\mathsf{S}_0$ be a connected closed semialgebraic subset of $\mathsf{H}_\infty(\R)$. \em Is there a polynomial (or a quasi-polynomial) map $f:\R^n\to\R^m$ such that $f(\R^n)_\infty=\mathsf{S}_0$?\em
\end{squestion}

For $m=2$ the answer is positive but for higher dimension we have no further information.
\begin{examples}\label{exxe}\em
For each connected closed semialgebraic subset $\mathsf{S}_0\subset\ell_\infty(\R)$ there exists a polynomial map $f:\R^2\to\R^2$ such that $\dim_\R(f(\R^2))=2$ and $(f(\R^2))_\infty=\mathsf{S}_0$. 

If $\mathsf{S}_0$ is not a singleton, the assertion follows from \cite[1.2]{u2}. On the other hand, the polynomial map $f:\R^2\to\R^2,\ (x,y)\mapsto(x,y^2+x^2)$ satisfies $f(\R^2))_\infty=\{(0:1:0)\}$, which is a singleton.
\end{examples}

\begin{sremark}\em
We have introduced quasi-polynomial maps to understand the limit of the image of a regular map to have a connected set of points at infinity. The following examples show that they do not enjoy the desired behavior.

(i) The composition of the quasi-polynomial maps 
\begin{multline*}
g:\R^2\to\R^2,\ (x,y)\mapsto(x,y^2)\\
\text{and}\quad f:\R^2\to\R^2,\ (x,y)\mapsto\Big(\frac{x^3}{1+x^2+y^2},\frac{y^3}{1+x^2+y^2}\Big)
\end{multline*} 
is not a quasi-polynomial map.

(ii) The image of the quasi-polynomial map $g:\R\to\R^2,\ t\mapsto(\frac{1}{1+t^2},1+t^2)$ is the semialgebraic set $S=\{xy=1, y\geq1\}$, which is not a polynomial image of $\R^n$.
\end{sremark}

\section{Set of points at infinity of a regular image of $\R^n$}\label{s4}

We have proved in Section \ref{s3} that the set of points at infinity of the image of a quasi-polynomial map $f:\R^n\to\R^m$ is connected. This is no longer true in general for regular maps even if $n=1$. 

\subsection{Preliminary examples}
We present some examples to illustrate the previous fact and to show that the conditions in the statement of Theorem \ref{property} are sharp.

\begin{examples}\label{ex1}\em
(i) The image of the regular map
$$
f:\R^2\to\R^2,\ (x,y)\mapsto\Big((xy-1)^2+x^2,\frac{1}{(xy-1)^2+x^2}\Big)
$$
is $S:=\{a>0,ab=1\}$, so $S_\infty=\{(0:1:0),(0:0:1)\}$ is disconnected.

(ii) The image of the regular map 
$$
f:\R^2\to\R^2,\ (x,y)\mapsto\Big(\frac{x^2}{1+y^2},\frac{y^2}{1+x^2}\Big)
$$
is $S:=\{a\geq0,b\geq0,ab<1\}$, so $S_\infty=\{(0:1:0),(0:0:1)\}$ is disconnected. If we write $f:=(\frac{f_1}{f_0},\frac{f_2}{f_0})$ where each $f_i$ is a non-zero polynomial, then $\deg(f_0)\!=\max\{\deg(f_1),\deg(f_2)\}$.

(iii) The image of the regular map
$$
f:\R^2\to\R^2,\ (x,y)\mapsto\Big(\frac{(1+x^4)y^6}{(1+y^4)^2},\frac{(1+y^4)x^4}{(1+x^4)^3}\Big)
$$
is a semialgebraic set $S$ such that $S_\infty=\{(0:1:0),(0:0:1)\}$ is disconnected. If we write $f:=(\frac{f_1}{f_0},\frac{f_2}{f_0})$ where each $f_i$ is a non-zero polynomial, then $\deg(f_0)\!<\max\{\deg(f_1),\deg(f_2)\}$. The set $Y_\C=\{(0:1:0),(0:0:1)\}$ of indeterminacy of the rational map
$$
F_\C:=(F_0:F_1:F_2):\C\PP^2\dashrightarrow\C\PP^2
$$ 
is contained in $\{F_0'=0\}\cap\R\PP^2$ where $F_0'=(\x_0^4+\x_1^4)^3(\x_0^4+\x_2^4)^2$ and $F_0=\x_0^2F_0'$.
\begin{proof}
Observe that $f(\R^2)\subset\{a^2b\leq1,a\geq0,b\geq0\}$ because
$$
\Big(\frac{(1+x^4)y^6}{(1+y^4)^2}\Big)^2\Big(\frac{(1+y^4)x^4}{(1+x^4)^3}\Big)=\Big(\frac{y^4}{1+y^4}\Big)^3\Big(\frac{x^4}{1+x^4}\Big)\leq 1.
$$
Thus, $S_\infty\subset\{(0:1:0),(0:0:1)\}$. To prove the converse inclusion it is enough to pick two rational paths $\alpha_i:(0,1]\to\R^2$ such that $\lim_{t\to0^+}\|\alpha_i(t)\|^2=+\infty$ and 
\begin{align*}
&\lim_{t\to0^+}\frac{g_1}{g_0}(\alpha_1(t))=+\infty,\ \lim_{t\to0^+}\frac{g_2}{g_0}(\alpha_1(t))=0,\\ 
&\lim_{t\to0^+}\frac{g_1}{g_0}(\alpha_2(t))=0,\ \lim_{t\to0^+}\frac{g_2}{g_0}(\alpha_2(t))=+\infty.
\end{align*}
For instance, $\alpha_1(\t):=(\frac{1}{\t},1)$ and $\alpha_2(t):=(1,\frac{1}{\t})$ do the job.
\end{proof}
\end{examples}

The following question arises naturally.

\begin{squestion}\em\label{q}
Given a closed semialgebraic subset $\mathsf{S}_0\subset\ell_\infty(\R)\subset\R\PP^2$: \em Is there a regular map $f:\R^2\to\R^2$ such that $(f(\R^2))_\infty=\mathsf{S}_0$? 
\end{squestion}

In case $\mathsf{S}_0$ is either connected or a finite set, the answer is by Examples \ref{exxe} positive.

\subsection{More sophisticated examples}\setcounter{paragraph}{0} 
If $\mathsf{S}_0$ is a finite set, we proceed as follows.

\begin{prop}
Let $H_i:=c_i\x-d_i\y$ be linear equations such that $c_id_i\neq0$ and the lines $\ell_i:=\{H_i=0\}$ are pairwise different. Denote $p_i:=(0:d_i^2:c_i^2)$. Then the image of the regular map
$$
h:=(h_1,h_2):\R^2\to\R^2,\ (x,y)\mapsto\Big(\frac{x^2}{1+\prod_{i=1}^rH_i(x,y)^2},\frac{y^2}{1+\prod_{i=1}^rH_i(x,y)^2}\Big)
$$
is a semialgebraic set $S$ such that $S_\infty=\{p_1,\ldots,p_r\}$.
\end{prop}
\begin{proof}
Observe first
$$
h\Big(\frac{d_i}{t},\frac{c_i}{t}\Big)=\Big(\frac{d_i^2}{t^2},\frac{c_i^2}{t^2}\Big)\equiv(t^2:d_i^2:c_i^2)\stackrel{t\to 0^+}{\longrightarrow}(0:d_i^2:c_i^2)=p_i,
$$
so $\{p_1,\ldots,p_r\}\subset S_\infty$. Next, we prove the converse inclusion $S_\infty\subset\{p_1,\ldots,p_r\}$. 

The Jacobian of $h$ is not identically zero, so $S:=h(\R^2)$ has dimension $2$. In addition, $S$ is pure dimensional because it is a regular image of $\R^2$. Thus,
$$
S':=S\setminus(h(\{x=0\})\cup h(\{y=0\})\cup h(\{\Jac(h)=0\}))
$$ 
is dense in $S$. Fix a point $p\in S_\infty$. After interchanging the variables and changing $\x$ with $-\x$ if necessary, there exists by Lemma \ref{polcurve} a rational path $\alpha:=(\t^{-k},\t^\ell{\rm p})$ where $k,\ell\in\Z$, $-k=\min\{-k,\ell\}<0$, ${\rm p}\in\R[\t]$ and ${\rm p}(0)\neq0$ such that $\lim_{t\to0^+}(h\circ\alpha)=p$ and $(h\circ\alpha)((0,\veps))\subset S'$ for $\veps>0$ small enough. Our change of variables does not modify the structure of $h$, so we keep the same notations. As $p\in\ell_\infty(\R)$, one of the following limits is infinity:
\begin{align*}
\lim_{t\to0^{+}}(h_1\circ\alpha)(t)&=\lim_{t\to0^{+}}\frac{t^{2k(r-1)}}{t^{2kr}+\prod_{i=1}^rH_i(1, t^{k+\ell}{\rm p}(t))^2},\\[6pt]
\lim_{t\to0^{+}}(h_2\circ\alpha)(t)&=\lim_{t\to0^{+}}\frac{t^{2kr+2\ell}{\rm p}(t)^2}{t^{2kr}+\prod_{i=1}^rH_i(1,t^{k+\ell}{\rm p}(t))^2}.
\end{align*}
Notice the following: 
\begin{itemize}
\item[(1)] The first limit is infinity if and only if 
$$
t^{2kr}+\prod_{i=1}^rH_i(1,\t^{k+\ell}{\rm p}(\t))^2=\t^{\nu_1}{\rm q}_1(\t)
$$ 
for some ${\rm q}_1\in\R[\t]$ and an integer $\nu_1\geq 2k(r-1)+1$. 
\item[(2)] The second limit is infinity if and only if $\ell<0$ and 
$$
t^{2kr}+\prod_{i=1}^rH_i(1,\t^{k+\ell}{\rm p}(\t))^2=\t^{\nu_2}{\rm q}_2(\t)
$$ 
for some ${\rm q}_2\in\R[\t]$ and an integer $\nu_2\geq 2kr+2\ell+1$.
\end{itemize}
As $k+\ell\geq0$, we deduce in both cases that there exists an index $i=1,\ldots,r$ such that $\lim_{t\to0^+}H_i(1,t^{k+\ell}{\rm p}(t))=0$. Since by hypothesis $c_i,d_i\neq0$, we conclude $\ell=-k<0$ and ${\rm p}(0)=\frac{c_i}{d_i}$. Denote $\mu:=\nu_j-2k(r-1)-1\geq0$ and ${\rm q}:={\rm q}_j$ in both cases, so
$$
\t^{2kr}+\prod_{i=1}^rH_i(1,\t^{k+\ell}{\rm p}(\t))^2=\t^{2k(r-1)+1}\t^\mu{\rm q}.
$$
We deduce
\begin{equation*}
\begin{split}
p&=\lim_{t\to0^+}(t^{2kr}+\prod_{i=1}^rH_i(1,{\rm p}(t))^2:t^{2k(r-1)}:t^{2kr+2\ell}{\rm p}(t)^2)\\
&=\lim_{t\to0^+}(t^{2k(r-1)+1}t^\mu{\rm q}(t):t^{2k(r-1)}:t^{2k(r-1)}{\rm p}(t)^2)\\
&=\lim_{t\to0^+}(t^{\mu+1}{\rm q}(t):1:{\rm p}(t)^2)=\Big(0:1:\frac{c_i^2}{d_i^2}\Big),
\end{split}
\end{equation*}
so $S_\infty\subset\{p_1,\ldots,p_r\}$. 
\end{proof}

We present next an example of a regular image $S$ such that $S_\infty$ has exactly two $1$-dimensional connected components.

\begin{lem}\label{2int}
There exists a regular map $f:\R^2\to\R^2$ whose image $S$ satisfies
$$
S_\infty=\Big\{(0:u:1):\ 0\leq u\leq \frac{1}{2}\Big\}\cup\Big\{(0:1:v):\ 0\leq v\leq \frac{1}{2}\Big\}.
$$ 
\end{lem} 
\begin{proof}
We build $f$ as the composition of two regular maps that we construct next:

\paragraph{}
\em Let $T:=\{0<a\leq1,b>0\}\cup\{0<b\leq1,a>0\}$. The image of the regular map 
$$
g:\R^2\to\R^2,\ (x,y)\mapsto\Big(\frac{x^2+1}{1+x^2y^2},\frac{y^2+1}{1+x^2y^2}\Big)
$$
is a semialgebraic set $S_1$ such that $T\subset S_1\subset T\cup[0,2]^2\subset\{a>0,b>0\}$\em. In particular, $S_{1,\infty}=\{(0:1:0),(0:0:1)\}$.

We check first $S_1\subset T\cup[0,2]^2$. Let $(x,y)\in \R^2$ and write $f_3(x,y)=:(a,b)$. We claim: \em If $b>2$, then $0<a\leq 1$\em. If $a>2$, we will have by symmetry $0<b\leq 1$, so $S_1\subset T\cup[0,2]^2$.

It is clear that $a>0$. Suppose by contradiction $a>1$. Then $x^2>x^2y^2$ and $y^2>1+2x^2y^2$, so
$$
x^2<x^2+2x^4y^2<x^2y^2<x^2,
$$
which is a contradiction. 

Next, we check $T\subset S_1$. It is enough to prove by symmetry that $T_1:=\{0<a\leq1,b>0\}\subset S_1$. Let $(a,b)\in T_1$ and consider the system of equations
$$
\left\{\begin{array}{l}
x^2+1=a(1+x^2y^2),\\[4pt]
y^2+1=b(1+x^2y^2).
\end{array}\right.\quad\leadsto\quad
\left\{\begin{array}{l}
b(x^2+1)-a(y^2+1)=0,\\[4pt]
ay^4+(a-1-b)y^2+b-1=0.
\end{array}\right.
$$
A simple discussion shows that both systems are equivalent. The discriminant $\Delta$ of the biquadratic equation $ay^4+(a-1-b)y^2+b-1=0$ is
$$
(a-1-b)^2-4a(b-1)=(b-3a+1)^2+8a(1-a),
$$
which is $\geq0$ because $0<a\leq1$. As $a-1-b<0$, the real number
$$
z_0:=\frac{b+(1-a)+\sqrt{(b-3a+1)^2+8a(1-a)}}{2a}
$$
is positive and has a square root $y_0$, which is a solution of the biquadratic equation 
$$
ay^4+(a-1-b)y^2+b-1=0. 
$$
The equation $b(x^2+1)-a(y_0^2+1)=0$ has a real solution $x_0$ if 
$$
0<2(a(y_0^2+1)-b)=2az_0+2a-2b=-b+1+a+\sqrt{(b-3a+1)^2+8a(1-a)}
$$
or equivalently if
$$
0<(b-3a+1)^2+8a(1-a)-(b-1-a)^2=4b(1-a).
$$
As $b>0$ and $a<1$, it holds $4b(1-a)>0$, so we deduce $(a,b)\in S_1$, as required.

\paragraph{}
Let $B_1:=\{0<2a\leq b\}$ and $B_2:=\{0<2b\leq a\}$. Write also $A_1:=\{0<x\leq 1,\ 4\leq y\}$ and $A_2:=\{0<y\leq 1,\ 4\leq x\}$. \em Then the image of $A:=A_1\cup A_2$ under the regular map
\begin{equation*}
\begin{split}
h:\R^2&\to\R^2,\\ 
(x,y)&\mapsto\Big(\frac{x((y-1)^2y^2+2(x-1)^2x)}{1+x(x-1)^2y(y-1)^2},\frac{y((x-1)^2x^2+2(y-1)^2y)}{1+x(x-1)^2y(y-1)^2}\Big),
\end{split}
\end{equation*}
is a semialgebraic set $S_2$ contained in $B:=B_1\cup B_2$, which satisfies
$$
S_{2,\infty}=B_{1,\infty}\cup B_{2,\infty}=\Big\{(0:u:1):\ 0\leq u\leq \frac{1}{2}\Big\}\cup\Big\{(0:1:v):\ 0\leq v\leq \frac{1}{2}\Big\}.
$$\em
Write $h:=(\frac{h_1}{h_0},\frac{h_2}{h_0})$ where
\begin{align*}
&h_0(\x,\y):=1+\x(\x-1)^2\y(\y-1)^2,\\
&h_1(\x,\y):=\x((\y-1)^2\y^2+2(\x-1)^2\x),\\
&h_2(\x,\y):=\y((\x-1)^2\x^2+2(\y-1)^2\y).
\end{align*} 
As $h_1(\y,\x)=h_2(\x,\y)$ and $h_0(\y,\x)=h_0(\x,\y)$, we have $\frac{h_1(\y,\x)}{h_0(\y,\x)}=\frac{h_2(\x,\y)}{h_0(\x,\y)}$, so it is enough to prove: \em $h(A_1)\subset B_1$ and $(h(A_1))_\infty=B_{1,\infty}$\em. 

Let $(x,y)\in A_1$. It holds $h_1(x,y)>0$ and
\begin{multline*}
h_2(x,y)-2h_1(x,y)=y((x-1)^2x^2+2(y-1)^2y)-2(x((y-1)^2y^2+2(x-1)^2x))\\
=2(1-x)(y-1)^2y^2+(y-4)(x-1)^2x^2\geq0
\end{multline*}
because $0<x\leq 1$ and $y\geq 4$, so $h(x,y)\in B_1$. Therefore, $(h(A_1))_\infty\subset B_{1,\infty}$ and it only remains to check $B_{1,\infty}\subset(h(A_1))_\infty$.

Indeed, for each $0<\lambda\leq 1$ consider the half-line $x=\lambda,y=t\geq4$ and the curve $C_\lambda\subset h(A_1)$ parametrized by
\begin{multline*}
\alpha_\lambda(t):=(\alpha_{\lambda1}(t),\alpha_{\lambda2}(t)):=g(\lambda,t)
=\Big(\frac{\lambda(t-1)^2t^2+2\lambda\mu_\lambda)}{1+\mu_\lambda t(t-1)^2},\frac{2(t-1)^2t^2+\lambda\mu_\lambda t}{1+\mu_\lambda t(t-1)^2}\Big)\\[-15pt]
\end{multline*}
where $\mu_\lambda:=\lambda(\lambda-1)^2$. As
$$
\lim_{t\to+\infty}\alpha_{\lambda1}(t)=+\infty,\qquad\lim_{t\to+\infty}\alpha_{\lambda2}(t)=+\infty\qquad\text{and}\qquad\lim_{t\to+\infty}\frac{\alpha_{\lambda1}(t)}{\alpha_{\lambda2}(t)}=\frac{\lambda}{2},
$$
we deduce $C_{\lambda,\infty}=\{(0:\frac{\lambda}{2}:1)\}$, so
$$
B_{1,\infty}=\bigcup_{0<\lambda\leq 1}C_{\lambda,\infty}\subset\Big(\bigcup_{0<\lambda\leq 1}C_\lambda\Big)_\infty=(h(A_1))_\infty,
$$
as required.

\paragraph{}
The image of the regular map $f:=h\circ g:\R^2\to\R^2$ is a semialgebraic set $S$ such that $S_\infty=\{(0:u:1):\ 0\leq u\leq \frac{1}{2}\}\cup\{(0:1:v):\ 0\leq v\leq \frac{1}{2}\}$, as required.
\end{proof}

\begin{squestion}\em
Let $\mathsf{S}_0$ be a closed semialgebraic subset of the hyperplane of infinity $\mathsf{H}_\infty(\R)$ of $\R\PP^m$. \em Is there a regular map $f:\R^n\to\R^m$ such that $(f(\R^n))_\infty=\mathsf{S}_0$?
\end{squestion}

\end{document}